\documentclass[english,11pt,leqno,twoside]{amsart}
\usepackage[utf8]{inputenc}
\usepackage[T1]{fontenc}
\usepackage{amsmath}
\usepackage{amsthm}
\usepackage{amssymb}
\usepackage{amsfonts}

\usepackage{soul,color}
\usepackage{t1enc}
\usepackage{a4,indentfirst,latexsym,graphics,graphicx}
\usepackage{mathrsfs}
\usepackage{enumitem}
\usepackage[colorlinks=true,urlcolor=blue,
citecolor=red,linkcolor=blue,linktocpage,pdfpagelabels,
bookmarksnumbered,bookmarksopen]{hyperref}
\usepackage[english]{babel}
\usepackage[left=2.1cm,right=2.1cm,top=2.72cm,bottom=2.72cm]{geometry}
\usepackage[metapost]{mfpic}
\usepackage[hyperpageref]{backref}
\usepackage[colorinlistoftodos]{todonotes}
\usepackage[normalem]{ulem}
\usepackage{braket}

\usepackage{mathtools}

\usepackage{bbm}
\usepackage[noadjust]{cite}
\usepackage{url}

\usepackage{lmodern}

\usepackage[useregional]{datetime2}

\usepackage{fixmath}

\numberwithin{equation}{section}
\newtheorem{Th}{Theorem}[section]
\newtheorem{Prop}[Th]{Proposition}
\newtheorem{Lem}[Th]{Lemma}

\newtheorem{Def}[Th]{Definition}
\newtheorem{Rem}[Th]{Remark}

\newcommand{\wt}{\widetilde}

\newcommand{\vp}{\varphi}
\newcommand{\eps}{\varepsilon}

\newcommand{\C}{\mathbb{C}}
\newcommand{\R}{\mathbb{R}}

\newcommand{\N}{\mathbb{N}}

\newcommand{\cA}{{\mathcal A}}

\newcommand{\cD}{{\mathcal D}}
\newcommand{\cE}{{\mathcal E}}
\newcommand{\cF}{{\mathcal F}}
\newcommand{\cG}{{\mathcal G}}
\newcommand{\cH}{{\mathcal H}}
\newcommand{\cI}{{\mathcal I}}
\newcommand{\cJ}{{\mathcal J}}

\newcommand{\cM}{{\mathcal M}}

\newcommand{\cS}{{\mathcal S}}

\newcommand{\weakto}{\rightharpoonup}

\newcommand{\supp}{\mathrm{supp}\,}
\newcommand{\dist}{\mathrm{dist}\,}

\newcommand{\loc}{\mathrm{loc}}

\newcommand{\rad}{\mathrm{rad}}
\newcommand{\tr}{\mathrm{tr}}
\newcommand{\divv}{\mathrm{div}}

\newcommand{\diam}{\mathrm{diam}}

\def\e{{\rm e}}
 \def\dd{\, {\rm d}}

\newcommand{\bessel}[2]{\mathcal{K}_{#1}(#2)}

\allowdisplaybreaks

\begin{document}
	
	\title[Fractional magnetic pseudorelativistic operator]{On a fractional magnetic pseudorelativistic operator: properties and applications}
	
	\author[F. Bernini]{Federico Bernini}
	\address[F. Bernini]{\newline\indent
		Dipartimento di Matematica
		\newline\indent
		Università degli Studi di Milano
		\newline\indent
		Via C. Saldini, 50
		\newline\indent
		20133 Milan, Italy.}
	\email{\href{mailto:federico.bernini@unimi.it}{federico.bernini@unimi.it}
	}
	
	\author[P. d'Avenia]{Pietro d'Avenia}
	\address[P. d'Avenia]{\newline\indent
		Dipartimento di Meccanica, Matematica e Management 
		\newline\indent
		Politecnico di Bari
		\newline\indent
		Via E. Orabona 4
		\newline\indent
		70125 Bari, Italy.}
	\email{\href{mailto:pietro.davenia@poliba.it}{pietro.davenia@poliba.it}
	}
	
	
	
	\begin{abstract}
		We introduce a fractional magnetic pseudorelativistic operator for a general {\em fractional order} $s\in(0,1)$. First we define a suitable functional setting and we prove some fundamental properties. Then we show the behavior of the operator as $s \nearrow 1$ obtaining some results {\em \`a la} Bourgain-Brezis-Mironescu and {\em removing the singularity} from the integral definition. Finally we get existence of weak solutions for some semilinear equations involving a power type nonlinearity or a nonlocal (Choquard type) term.
	\end{abstract}
	
	\maketitle
	
	\begin{center}
		\begin{minipage}{11cm}
			\tableofcontents
		\end{minipage}
	\end{center}

	\section{Introduction}
	In the last years a great attention has been devoted to {\em fractional} operators and, in particular, to {\em nonlocal PDEs} involving such type of operators, due to several models and applications in which they appear.
	
	From a mathematical point of view, fractional operators can be obtained as {\em generalizations} of the corresponding local ones.\\
	Indeed, it is well known that, if $u\in \mathcal{S}(\R^3)$, the Schwartz space of rapidly decaying functions,
	\[
	-\Delta u = \mathcal{F}^{-1}(|\xi|^{2}\mathcal{F}(u)(\xi)),
	\]
	where $\mathcal{F}$ denotes the Fourier transform and the function $\xi\mapsto |\xi|^2$ is usually called {\em symbol}.
	
	Its fractional version, the {\em Fractional Laplacian}, can be defined generalizing the symbol, namely taking the function $\xi\mapsto |\xi|^{2s}$, $s\in(0,1)$, and so
	\begin{equation}
		\label{fraclapl}
		(-\Delta)^s u:= \mathcal{F}^{-1}(|\xi|^{2s}\mathcal{F}(u)(\xi)),
		\qquad
		s\in(0,1),
		u\in \mathcal{S}(\R^3).
	\end{equation}
	The Fractional Laplacian is a {\em nonlocal} operator and it can be proved that, for $u\in \mathcal{S}(\R^3)$,
	\begin{equation}
		\label{slaplacian}
		(-\Delta)^s u (x) = c_s\lim_{\eps\searrow 0}\int_{B^c_\eps(x)}\frac{u(x)-u(y)}{|x-y|^{3+2s}} \dd y,
		\quad x\in\R^3,
	\end{equation}
	where 
	\[
	c_{s} 
	= s 2^{2s} \frac{\Gamma\big(\frac{3+2s}{2}\big)}{\pi^{3/2}\Gamma(1-s)}
	= 
	\left(\int_{\R^3} \frac{1-\cos(\zeta_1)}{|\zeta|^{3+2s}} \dd \zeta \right)^{-1} 
	\]
	and $\Gamma$ is the Gamma function.\\ 
	The integral representation \eqref{slaplacian} can be also obtained by the L\'evy-Khintchine formula, considering the generator ${\mathscr H}$ of the semigroup on $C^\infty_c(\R^3)$ associated to a general L\'evy process which is given by
	\begin{equation}\label{LKs}
		\begin{split}
			{\mathscr H}u(x)
			&= -a_{ij}\partial_{x_i x_j}^2 u(x)-b_i\partial_{x_i}u(x)\\
			&\quad -\lim_{\eps\searrow 0}\int_{B^c_\eps(0)}\Big(u(x+y)-u(x)-1_{\{|y|<1\}}(y)y\cdot\nabla u(x)\Big)\dd\mu,
		\end{split}
	\end{equation}
	where $\mu$ is a L\'evy nonnegative measure, namely 
	\[
	\int_{\R^3}\frac{|y|^2}{1+|y|^2}\dd\mu<+\infty,
	\]
	and the integral term represents the purely jump part of the L\'evy process.\\
	Then, taking, among the L\'evy processes, the only stochastically stable ones having jump part, namely
	\[
	\dd\mu=\frac{c_s}{|y|^{3+2s}}\dd y,
	\]
	the last term in \eqref{LKs} reads as the right hand side of \eqref{slaplacian}.\\
	For more details the reader can refer to \cite{DiPaVa}.\\
	Moreover, up to multiplicative constants, the {\em fractional counterpart} of $\|\nabla u\|_2^2$ is the {\em Gagliardo seminorm}
	\[
	\int_{\R^3\times\R^3} \frac{|u(x)-u(y)|^2}{|x-y|^{3+2s}} \dd x\dd y 
	\]
	and Bourgain, Br\'ezis, and Mironescu in \cite{BoBrMi} proved that, if $\Omega$ is a smooth bounded domain and $u\in W^{1,2}(\Omega)$,
	\begin{equation}
		\label{BBM:classic}
		\lim_{s\nearrow 1} (1-s) \int_{\Omega\times\Omega} \frac{|u(x)-u(y)|^2}{|x-y|^{3+2s}} \dd x\dd y 
		=\frac{2}{3}\pi \int_\Omega |\nabla u|^2 \dd x.
	\end{equation}
	The same result holds for $\Omega=\R^3$ (see \cite[Remark 4.3]{DiPaVa}).\footnote{Often in the literature the coefficient of the integral in the rhs of \eqref{BBM:classic} is written as
		\[
		Q_3=\frac12\int_{\mathbb{S}^2}|e \cdot k|^2 \dd \sigma(k),   
		\]
		where $\mathbb{S}^2$ is the unit sphere in $\mathbb{R}^3$ and $e$ is an arbitrary element of $\mathbb{S}^2$.
	}
	
	It is usual to meet PDEs where the sum of the term $-\Delta u$ or $(-\Delta)^s u$ plus $m^2 u$ or $m^{2s} u$ (sometimes called {\em mass} term) appears.
	
	On the other hand, in many applications, especially in Physics for questions related to pseudo-relativistic boson stars (see e.g. \cite{ElgartSchlein,FrolichJonssonLenzmann,LiebThirring,LiebYau87,LiebYau88}), such a {\em perturbation} of the {\em differential} operator with a mass term {\em acts} in a different form: taking the {\em symbol} $\sqrt{|\xi|^2+m^2}$, the operator
	\begin{equation}
		\label{pseudorelop}
		\sqrt{-\Delta + m^2} u
		:=\mathcal{F}^{-1}(\sqrt{|\xi|^{2}+m^2}\mathcal{F}(u)(\xi)),
		\qquad
		u\in \mathcal{S}(\R^3),
	\end{equation}
	appears (see also e.g. \cite{LiLo,LiebSeiringer})\footnote{It is usual to find the operator in \eqref{pseudorelop} in the form $(\sqrt{-\Delta + m^2}-m) u$. Here we focus in the nontrivial part \eqref{pseudorelop}.}.\\
	Observe that, if $m=0$, the operator \eqref{pseudorelop} coincides with \eqref{fraclapl} for $s=1/2$.
	
	Of course it is natural to generalize the symbol in \eqref{pseudorelop} considering a general power $s\in(0,1)$, namely
	\begin{equation}
		\label{pseudo:nonmagn:s:fourier}
		(-\Delta + m^2)^s u
		:=
		\mathcal{F}^{-1}(({|\xi|^{2}+m^2)^s}\mathcal{F}(u)(\xi)),
		\qquad
		u\in \mathcal{S}(\R^3),
		\quad
		s\in(0,1)
	\end{equation}
	(see \cite{CarmonaMastersSimon}).\\
	It is possible to prove (see e.g. \cite[Remark 5]{FaFe}) that, for $u\in C^\infty_c(\R^3)$,
	\begin{equation}
		\label{pseudo:nonmagn:s}
		(-\Delta+m^2)^su(x)= C_s m^{\frac{3+2s}{2}}\lim_{\eps\searrow 0}\int_{B^c_\eps(x)} \frac{u(x)-u(y)}{|x-y|^{\frac{3+2s}{2}}}\mathcal{K}_{\frac{3+2s}{2}}(m|x-y|) \dd y + m^{2s}u(x),
		\quad
		x\in\R^3,
	\end{equation}
	where
	\begin{equation}
		\label{pse:rel:mag:con}
		C_s := c_s\frac{2^{-\frac{3+2s}{2}+1}}{\Gamma\left(\frac{3+2s}{2} \right)} = \frac{s2^{\frac{2s-1}{2}}}{\pi^{\frac{3}{2}}\Gamma(1-s)}
	\end{equation}
	and $\mathcal{K}_\nu$ denotes the \textit{modified Bessel function of the third kind\footnote{In literature, it can also be found under other names, like modified Bessel function of the second kind of order $\nu$,  Basset function or MacDonald function (e.g. see \cite{ErMaObTr}). } of order $\nu$}.
	
	\begin{Rem}
		If $m\geq 0$ and $u\in C^\infty_c(\R^3)$, simply starting from the integral term of \eqref{LKs} and taking
		\begin{equation}
			\label{fract:Levy:measure}
			\dd\mu_s^m=\mu_s^m(y)\dd y = 
			\begin{cases}
				\displaystyle \frac{c_s}{|y|^{3+2s}}\dd y, &m=0,\\
				\displaystyle C_s m^{\frac{3+2s}{2}}\frac{\mathcal{K}_{\frac{3+2s}{2}}(m|y|)}{|y|^{\frac{3+2s}{2}}}\dd y, &m>0,
			\end{cases}
		\end{equation}
		we can write both \eqref{slaplacian} and \eqref{pseudo:nonmagn:s} as
		\[
		(-\Delta + m^2)^s u (x)
		=
		\lim_{\eps\searrow 0} \int_{B^c_\eps(x)} [u(x)-u(y)] \mu_s^m(x-y) \dd y + m^{2s}u(x),
		\quad
		x\in\R^3.
		\]
	\end{Rem}
	
	Also in this case, in the spirit of \cite{BoBrMi}, Ambrosio and Essebei in \cite{AmEs2024}, taking for simplicity $m=1$, proved that
	\begin{equation}
		\label{AE:bounded}
		\lim_{s\nearrow 1} \frac{C_s}{2} \int_{\Omega\times\Omega} \frac{|u(x)-u(y)|^2}{|x-y|^\frac{3+2s}{2}} \mathcal{K}_{\frac{3+2s}{2}}(|x-y|) \dd x\dd y 
		=\int_\Omega |\nabla u|^2 \dd x
	\end{equation}
	for $\Omega\subset\R^3$ open, bounded, and with Lipschitz boundary, and
	\begin{equation}
		\label{AE:unbounded}
		\lim_{s\nearrow 1} \frac{C_s}{2} \int_{\R^3\times\R^3} \frac{|u(x)-u(y)|^2}{|x-y|^\frac{3+2s}{2}} \mathcal{K}_{\frac{3+2s}{2}}(|x-y|) \dd x\dd y 
		=\int_{\R^3} |\nabla u|^2 \dd x.
	\end{equation}
	
	\begin{Rem}
		Observe that, since
		\begin{equation*}
			\lim_{s \nearrow 1} \frac{c_s}{1-s} 
			= \frac{3}{\pi}
		\end{equation*}
		and
		\begin{equation*}
			\lim_{s \nearrow 1} \frac{C_s}{1-s} = \frac{\sqrt{2}}{\pi^{3/2}},
		\end{equation*}
		the results {\em \`a la} \cite{BoBrMi} can be expressed equivalently, up to a constant, considering in the left hand side the coefficient $(1-s)$ or, respectively, $c_s$ or $C_s$.\footnote{From \eqref{AE:bounded} and \eqref{AE:unbounded}, in such kind of results we will use the constants that appear to us to be simplest.}
	\end{Rem}
	
	Now, let us consider functions $u:\R^3\to\C$. If $u$ represents a charged matter field, to study its interaction with an assigned magnetic field $B=\nabla\times A$, with $A:\R^3\to\R^3$, in the {\em local} case it is usual to start from the Lagrangian density of the {\em free} particle and to replace the spatial derivative $\nabla$ with the {\em covariant derivative} $\nabla_A:=\nabla - i A$, where $i$ is the imaginary unit (see e.g. \cite{Felsager,Naber:Int,Naber:Found})\footnote{Actually we should take $\nabla - i q A$, where $q$ is a coupling constant. Here, for simplicity, we take $q=1$.}. Thus, the Lagrangian density term $|\nabla u|^2/2$, which gives arise to $-\Delta u$ term in the Euler-Lagrange equation, becomes $|(\nabla -iA) u|^2/2$ and, when we consider the Euler-Lagrange equation of such a term, we get the so called {\em magnetic Laplacian}
	\begin{equation}
		\label{local:magn:Laplacian}
		-(\nabla - iA)^2u = -\Delta u + iu\divv A + 2 i A \cdot \nabla u  + |A|^2 u
	\end{equation}
	(see also \cite{AvronHerbstSimon,ReedSimon1980}).
	
	At this point it appears natural to consider a relativistic version of the magnetic Laplacian and generalize it for $s\in(0,1)$. This is not only a mathematical trick but, for instance, the relativistic magnetic Laplacian appears too in several applications (see e.g. \cite{Bergmann,LiebSeiringer,FrankLiebSeiringer}).
	
	Following \cite{BerezinSubin,GrossmannLoupiasStein,Hormander1985}, in \cite{IchinoseTamura} Ichinose and Tamura introduced, through oscillatory integrals, the {\em Weyl pseudodifferential operator defined with mid-point prescription}
	\begin{equation}
		\label{Weyl:operator}
		\begin{aligned}
			\cH_A^{1/2} u (x)
			&= \frac{1}{(2\pi)^3}\int_{\R^3 \times \R^3}\e^{i(x-y)\cdot \xi}\sqrt{\left|\xi-A\left(\frac{x+y}{2}\right)\right|^2 + m^2}u(y) \dd y \dd \xi\\
			&= \frac{1}{(2\pi)^3}\int_{\R^3 \times \R^3}\e^{i(x-y)\cdot \left(\xi+A\left(\frac{x+y}{2}\right)\right)}\sqrt{|\xi|^2 + m^2}u(y) \dd y \dd \xi
		\end{aligned}
	\end{equation}
	as a relativistic version of the magnetic Laplacian, proving also that, for $u \in C^{\infty}_c(\R^3)$, \eqref{Weyl:operator} is equivalent to 
	\begin{equation}
		\label{pseudo:magn:12}
		\begin{aligned}
			\cH_A^{1/2} u(x) 
			&= - \lim_{\eps \searrow 0} \int_{B^c_\eps(0)} \left(\e^{-iy\cdot A\left(x + \frac{y}{2}\right)}u(x+y)-u(x) - 1_{\{|y|<1\}}(y)y\cdot (\nabla_x-iA(x))u(x)\right) \dd \mu_{1/2}^m\\
			&\qquad + mu(x) \\
			&=  \lim_{\eps \searrow 0} \int_{B^c_\eps(x)} \left(u(x)-\e^{i(x-y)\cdot A\left(\frac{x+y}{2}\right)}u(y)\right) \mu_{1/2}^m(y-x) \dd y 
			+mu(x),
		\end{aligned}
	\end{equation}
	where $\mu_{1/2}^m$ is given in \eqref{fract:Levy:measure} taking $s=1/2$, and obtaining in the first line a formulation which is analogous to \eqref{LKs}.
	For the detailed construction, see also \cite[Section 3.1]{Ichinose}.
	
	For the sake of completeness, we report that in the literature there are also other definitions for the pseudorelativistic magnetic operator (see e.g. \cite{IftimieMantoiuPurice2007,LiebSeiringer}).
	
	The previous considerations show as it appears natural to generalize such a relativistic magnetic operator considering in \eqref{pseudo:magn:12} the density measure $\mu_s^m$ for $s\in(0,1)$, namely
	\[
	\cH_A^{s} u(x)
	=  \lim_{\eps \searrow 0} \int_{B^c_\eps(x)} \left(u(x)-\e^{i(x-y)\cdot A\left(\frac{x+y}{2}\right)}u(y)\right) \mu_{s}^m(y-x) \dd y 
	+m^{2s} u(x).
	\]
	
	In \cite{dASq} the case $m=0$ has been studied.\footnote{We point out that in \cite{FrankLiebSeiringer} the operator $|-i\nabla - A|^{2s} := ((-i\nabla - A)^2)^s$ has been defined by the spectral theorem for $0 < s \leq 1$.} Here, for the operator
	\begin{equation}
		\label{dASq:operator}
		(-\Delta)^s_Au(x):= c_s \lim_{\eps \searrow 0} \int_{B^c_\eps(x)} \frac{u(x)-\e^{i(x-y)\cdot A\left(\frac{x+y}{2}\right)}u(y)}{|x-y|^{3+2s}} \dd y,
	\end{equation}
	the authors constructed a suitable functional setting showing its properties and they obtained existence results for some variational nonlinear equations involving \eqref{dASq:operator}. Moreover, in \cite[Theorem 2.6]{dASq}, they showed that, for {\em smooth} data, {\em smooth weak solutions} of
	\[
	(-\Delta)^s_Au + u=f
	\qquad
	\text{in }\R^3
	\]
	solves the same equation pointwise a.e. in $\R^3$.\\
	Later on Squassina and Volzone in \cite{SqVo} and Nguyen, Pinamonti, Squassina, and Vecchi \cite{NgPiSqVe}, in the spirit of \cite{BoBrMi}, proved that
	\[
	\lim_{s\nearrow 1} \frac{c_s}{2} \int_{\Omega\times\Omega} \frac{|u(x)-\e^{i(x-y)\cdot A\left(\frac{x+y}{2}\right)}u(y)|^2}{|x-y|^{3+2s}} \dd x\dd y 
	= \int_\Omega |(\nabla-iA) u|^2 \dd x
	\]
	for $\Omega\subset\R^3$ open, bounded, with Lipschitz boundary, and $A\in C^2(\overline{\Omega},\R^3)$, and for $\Omega=\R^3$ and $A$ Lipschitz, respectively.\\
	Starting from \cite{dASq}, a wide literature has been developed considering in particular several elliptic type equations, several types of nonlinearities, and some generalizations of the operator \eqref{dASq:operator}.
	
	Here we want to complete the picture considering the case $m>0$, namely the operator
	\begin{equation}
		\label{pseudo:mag:Sch:order:s}
		\begin{split}
			(-\Delta + m^2)_A^s u(x)
			&:= C_sm^{\frac{3+2s}{2}} \lim_{\eps \searrow 0}\int_{B^c_{\eps}(x)}\frac{u(x) - \e^{i(x-y)\cdot A\left(\frac{x+y}{2}\right)}u(y)}{|x-y|^{\frac{3+2s}{2}}} \mathcal{K}_{\frac{3+2s}{2}}(m|x-y|) \dd y \\
			& \qquad
			+ m^{2s}u(x),
			\qquad
			x\in\R^3.
		\end{split}
	\end{equation}
	
	The results presented in this paper focus mainly on three directions.
	
	Our first purpose is to define a functional setting for the operator \eqref{pseudo:mag:Sch:order:s} and give some properties that will be useful in the rest of the paper (see Section \ref{Functional:setting:sec}).
	
	Then, we consider the behavior of \eqref{pseudo:mag:Sch:order:s} and of the norm defined in Section \ref{Functional:setting:sec} as the \textit{fractional exponent} $s\nearrow 1$.\\
	In particular, in the spirit of \cite{BoBrMi,SqVo}, if
	\[
	H^1_A(\R^3,\C) := \left\{u \in L^2(\R^3,\C) : \left(\int_{\R^3} \left|(\nabla -iA)u\right|^2  \dd x \right)^{\frac12} < +\infty\right\},
	\]
	we prove
	\begin{Th}
		\label{main:result:BBM:bounded}
		Let $\Omega \subset \R^3$ be an open bounded set with Lipschitz boundary and $A \in C^2(\overline{\Omega},\R^3)$. Then, for every $u \in H^1_A(\Omega,\C)$  
		\begin{equation*}
			\lim_{s \nearrow 1}\frac{C_s}{2}m^{\frac{3+2s}{2}}\int_{\Omega \times \Omega} \frac{\left|u(x)-\e^{i(x-y)\cdot A\left(\frac{x+y}{2}\right)}u(y) \right|^2}{|x-y|^{\frac{3+2s}{2}}}\mathcal{K}_{\frac{3+2s}{2}}(m|x-y|) \dd x \dd y = \int_{\Omega}|\nabla_Au|^2 \dd x.
		\end{equation*}
	\end{Th}
	\noindent 
	Following \cite{NgPiSqVe}, we observe that a similar result can be found also in the whole space setting.
	\begin{Th}
		\label{main:result:BBM:unbounded}
		Let $A:\R^3 \to \R^3$ be Lipschitz. Then, for every $u \in H^1_A(\R^3,\C)$  
		\begin{equation*}
			\lim_{s \nearrow 1}\frac{C_s}{2}m^{\frac{3+2s}{2}}\int_{\R^3 \times \R^3} \frac{\left|u(x)-\e^{i(x-y)\cdot A\left(\frac{x+y}{2}\right)}u(y) \right|^2}{|x-y|^{\frac{3+2s}{2}}}\mathcal{K}_{\frac{3+2s}{2}}(m|x-y|) \dd x \dd y 
			= 
			\int_{\R^3}|\nabla_Au|^2 \dd x.
		\end{equation*}
	\end{Th}
	
	Moreover, in the spirit of \cite[Proposition 4.4]{DiPaVa} (see also \cite{AmBuMePe}), we provide the asymptotic behavior of \eqref{pseudo:mag:Sch:order:s} as $s$ goes to $1$, namely 
	\begin{Th}
		\label{ABMP:Th}
		Let $A \in C^2(\R^3,\C)$, $m>0$. For every $u \in C^{\infty}_c(\R^3,\C)$ it holds
		\begin{equation*}
			\lim_{s \nearrow 1} (-\Delta + m^2)_A^su = (-\Delta_A + m^2)u
			\quad \text{ in }\R^3.
		\end{equation*}
	\end{Th}

	Later, in Section \ref{section:pointwise:representation}, we give a suitable definition of weak solution and we show that {\em smooth} weak solutions of
	\[
	(-\Delta + m^2)_A^s u= f \quad \text{ in } \R^3
	\]
	solve it pointwise a.e. in $\R^3$, for suitable potentials $A$'s and right hand sides $f$'s (cf. Theorem \ref{th:pointwise:representation}).
	
	Finally, the last part of the paper is devoted to some applications of classical variational methods in the setting defined in Section \ref{Functional:setting:sec} that allow to solve some nonlinear equations involving \eqref{pseudo:mag:Sch:order:s}.\\
	More specifically, under suitable assumptions on the potential $A$, we  find \textit{ground state} solutions to the semilinear equation involving a power type nonlinearity
	\begin{equation}
		\label{eq:semilinear}
		(-\Delta + m^2)_A^su = |u|^{p-2}u \quad \text{ in } \R^3,
	\end{equation}
	where $s \in (0,1)$, $p \in (2,2^*_s)$ and $2^*_s = 6/(3-2s)$ is the fractional critical Sobolev exponent for $N=3$, and to the Choquard type equation
	\begin{equation}
		\label{eq:Choquard}
		(-\Delta + m^2)_A^su = (I_{\alpha} * |u|^p)|u|^{p-2}u \quad  \text{ in } \R^3,
	\end{equation}
	for $p \in (1+\alpha/3,(3+\alpha)(3-2s))$, where $I_{\alpha}:\R^3 \to \R$ is the Riesz potential defined as
	\[
	I_{\alpha}(x)=\pi^{3/2}2^{\alpha}\frac{\Gamma\left(\frac{\alpha}{2}\right)}{\Gamma\left(\frac{3-\alpha}{2}\right)}\frac{1}{|x|^{3-\alpha}}, \quad \alpha \in (0,3).
	\]
	
	The solutions will be found by solving suitable minimization problems, both for \eqref{eq:semilinear} and \eqref{eq:Choquard}, in symmetric and general settings (see Section \ref{Applications}).
	
	Starting from the seminal paper \cite{EstebanLions}, that involves the {\em local} operator \eqref{local:magn:Laplacian} (see also \cite{ArioliSzulkin2003,CiClSe2012,BoNyVS2019} and references therein), nonlinear equations involving magnetic operators has been widely studied by variational methods. In particular, in the last years, a great attention has been devoted to equations involving the fractional magnetic operator \eqref{dASq:operator} introduced in \cite{dASq}, in many different cases (see e.g. \cite{AffiliValdinoci,Ambrosio2019-3,FiPiVe} and references therein). Also in the pseudorelativistic case the literature is wide both for $s=1/2$ (see e.g. \cite{BieganowskiSecchi2019,BeBiSe2022,ChoiSeok,CiSe2015-2,CotiZelatiNolasco2011,GuoZeng,Lenzmann2009} and references therein) and for $s \in (0,1)$ (see \cite{Ambrosio2016,BuenoMiyagakiPereira,FaFe,Secchi2019} and references therein). Very few papers has been devoted to the magnetic case $m>0$. Here we mention \cite{CiSe2018} for $s=1/2$.
	
	As far as we know, this paper represents the first approach to this type of pseudodifferential magnetic operators for a general $s\in(0,1)$.\\
	With respect to the existing literature, the main technical differences are the following.\\
	Of course, comparing the operator \eqref{pseudo:mag:Sch:order:s} with \eqref{pseudo:nonmagn:s:fourier}, we highlight the absence of a definition through the Fourier transform, which sometimes is very important (see \cite{FaFe}) and the presence of the magnetic potential $A$ seems to require to consider only the definitions through the integral representation, which is analogous to \eqref{pseudo:nonmagn:s}.\\
	In addition, the presence in \eqref{pseudo:mag:Sch:order:s} of the Bessel function\footnote{For completeness, we report in Appendix \ref{Appendix:Bessel} the properties of the Bessel function we will use throughout the paper.}, that does not appear in \eqref{dASq:operator}, adds a further element of complexity.
	
	We conclude observing that, for the case $s=1/2$, Cingolani and Secchi in \cite{CiSe2018} considered a magnetic relativistic Hartree (or Choquard) type equation. Their approach differs from ours, indeed they defined the operator as the square root of the {\em local} magnetic case in the same functional setting by \cite{dASq}, and they obtained existence results in a symmetric setting for $A$ continuous and bounded or linear passing from a Caffarelli-Silvestre type extension, i.e. a local realization of the square root of the magnetic laplacian to a local elliptic operator with Neumann boundary condition on a half-space.

	In what follows, $C$ denotes a generic positive constant which may vary from one line to another. Moreover, since we are interested in the physical case, namely in three-dimensional space, as done in this introduction, when we will recall results already proved for an arbitrary dimension $N$, we will report their statements for $N=3$.

	\section{Functional setting}
	\label{Functional:setting:sec}
	
	This section is devoted to our functional setting and it is divided into three subsections. Following \cite{dASq}, in the first one we introduce  the fractional Sobolev space $H^s_A(\R^3,\C)$ in which we {\em define} the operator \eqref{pseudo:mag:Sch:order:s}. Then, we provide some features of spaces without a magnetic field. These results will be necessary for the last subsection, where we prove properties for the space $H^s_A(\R^3,\C)$ which will be used in the rest of the paper.
	
	\subsection{The space \texorpdfstring{$H^s_A(\R^3,\C)$}{HsA(R3,C)}}
	Let $L^2(\R^3,\C)$ be the Lebesgue space of complex-valued functions endowed with the inner product
	\[
	(u,v)_{L^2(\R^3,\C)} = \Re\int_{\R^3} u(x)\overline{v(x)} \dd x.
	\]
	Moreover, if $A:\R^3 \to \R^3$ is a continuous field, let us consider 
	\[
	\cH:=\left\{u:\R^3 \to \C : (u,u)_{A,s} < +\infty \right\}
	\]
	where 
	\begin{equation}
		\label{scalar:product}
		(u,v)_{A,s}:= \frac{C_s}{2}m^{\frac{3+2s}{2}} \langle u,v\rangle_{A,s} + m^{2s} (u,v)_{L^2(\R^3,\C)}
	\end{equation}
	and
	\[
	\langle u,v\rangle_{A,s}
	:=
	\Re\int_{\R^3 \times \R^3} \frac{\left(\e^{-i(x-y)\cdot A\left(\frac{x+y}{2}\right)}u(x)-u(y) \right)\overline{\left(\e^{-i(x-y)\cdot A\left(\frac{x+y}{2}\right)}v(x)-v(y) \right)}}{|x-y|^{\frac{3+2s}{2}}}\mathcal{K}_{\frac{3+2s}{2}}(m|x-y|) \dd x \dd y.
	\]
	\\
	Clearly, \eqref{scalar:product} is a real scalar product on $\cH$ and, hence,
	\begin{equation}
		\label{magn:norm}
		\|u\|_{A,s}:=\sqrt{(u,u)_{A,s}^2}
	\end{equation}
	is a norm.
	\\
	In the following, we will also use, for simplicity, the notation
	\begin{equation}
		\label{Gagliardo:seminorm}
		[u]_{A,s}:=\sqrt{\langle u,u \rangle^2_{A,s}}.
	\end{equation}
	We can refer to \eqref{Gagliardo:seminorm} as the \textit{pseudorelativistic magnetic Gagliardo seminorm}.
	
	We have
	\begin{Prop}
		\label{Prop:scalar:prod:Hilb:space}
		$(\cH,(\cdot,\cdot)_{A,s})$ is a real Hilbert space.
	\end{Prop}
	\begin{proof}
		Let $\{u_n\}_n \subset \cH$ be a Cauchy sequence, that is 
		\begin{equation}
			\label{Cauchy:seq}
			\forall \eps>0 \, \, \exists \bar{n}(\eps) \in \N \text{ such that } \forall n,m \geq \bar{n}(\eps) : \|u_n-u_m\|_{A,s} < \eps.
		\end{equation}
		It follows that $\{u_n\}_n$ is a Cauchy sequence in $L^2(\R^3,\C)$, so that there exists $u \in L^2(\R^3,\C)$ such that $u_n \to u$ in $L^2(\R^3,\C)$ and a.e. in $\R^3$.
		\\
		We claim that $u \in \cH$. From Fatou's Lemma, assuming, without loss of generality, $\eps=1$ in \eqref{Cauchy:seq}, it follows that
		\[
		\frac{C_s}{2}m^{\frac{3+2s}{2}}[u]^2_{A,s} \leq \frac{C_s}{2}m^{\frac{3+2s}{2}}\liminf_n [u_n]^2_{A,s} \leq \liminf_n  \|u_n - u_{\overline{n}(1)}\|^2_{A,s} + \|u_{\overline{n}(1)}\|^2_{A,s} \leq 1 + \|u_{\overline{n}(1)}\|^2_{A,s} < + \infty.
		\]
		Hence, we can conclude by observing that, from \eqref{Cauchy:seq}, for every $\eps>0$ and $n$ large,
		\[
		[u_n - u]_{A,s} \leq \liminf_k [u_n - u_k]_{A,s}< \eps.
		\]
	\end{proof}

	From now on we will assume that $A$ has locally bounded gradient.
	
	The space $\cH$ satisfies also the following property.
	\begin{Prop}
		$C^{\infty}_c(\R^3,\C)$ is a subspace of $\cH$.
	\end{Prop} 
	\begin{proof}
		It suffices to prove that $[u]_{A,s} < +\infty$ for all $u \in C^{\infty}_c(\R^3,\C)$.
		Let $\supp u$ be the compact support of $u$, therefore
		\begin{multline*}
			\int_{\R^3 \times \R^3} \frac{\left|\e^{-i(x-y)\cdot A\left(\frac{x+y}{2}\right)}u(x)-u(y) \right|^2}{|x-y|^{\frac{3+2s}{2}}}\mathcal{K}_{\frac{3+2s}{2}}(m|x-y|) \dd x \dd y \\
			\leq 2\int_{\supp u \times \R^3} \frac{\left|\e^{-i(x-y)\cdot A\left(\frac{x+y}{2}\right)}u(x)-u(y) \right|^2}{|x-y|^{\frac{3+2s}{2}}}\mathcal{K}_{\frac{3+2s}{2}}(m|x-y|) \dd x \dd y.
		\end{multline*}
		Since the function
		\[
		(x,y) \mapsto \e^{-i(x-y)\cdot A\left(\frac{x+y}{2}\right)}u(x)
		\]
		is bounded on $\supp u \times \R^3$ and $A$ has locally bounded gradient, from the Lagrange Theorem it follows that
		\[
		|\e^{-i(x-y)\cdot A\left(\frac{x+y}{2}\right)}u(x) - u(y)| \leq C|x-y|
		\]
		for any $(x,y) \in \supp u \times \R^3$. Moreover, we have that $$|\e^{-i(x-y)\cdot A\left(\frac{x+y}{2}\right)}u(x) - u(y)| \leq C$$ for every $(x,y) \in \supp u \times \R^3$. Hence, using \eqref{Bessel:bound},
		\begin{align*}
			&\int_{\supp u \times \R^3} \frac{\left|\e^{-i(x-y)\cdot A\left(\frac{x+y}{2}\right)}u(x)-u(y) \right|^2}{|x-y|^{\frac{3+2s}{2}}}\mathcal{K}_{\frac{3+2s}{2}}(m|x-y|) \dd x \dd y \\
			& \qquad \qquad \leq C \int_{\supp u \times \R^3} \frac{\min\{|x-y|^2,1\}}{|x-y|^{\frac{3+2s}{2}}}\mathcal{K}_{\frac{3+2s}{2}}(m|x-y|) \dd x \dd y \\
			& \qquad \qquad \leq C \int_{B_1(0)} \frac{1}{|z|^{\frac{2s-1}{2}}}\mathcal{K}_{\frac{3+2s}{2}}(m|z|) \dd z + C \int_{B^c_1(0)} \frac{1}{|z|^{\frac{3+2s}{2}}}\mathcal{K}_{\frac{3+2s}{2}}(m|z|) \dd z \\
			& \qquad \qquad \leq C \int_{B_1(0)} \frac{1}{|z|^{1+2s}} \dd z + C \int_{B^c_1(0)} \frac{1}{|z|^{3+2s}} \dd z < + \infty.
		\end{align*}
	\end{proof}
	
	Therefore, we define the \textit{pseudorelativistic fractional magnetic Sobolev} space $H^s_{A}(\R^3,\C)$ as the closure of $C^{\infty}_c(\R^3,\C)$ in $\cH$.
	
	\subsection{Some remarks on \texorpdfstring{$H^s$}{Hs} and \texorpdfstring{$H^s_m$}{Hsm}}
	\label{Hs:Hsm:remarks}
	In this subsection, we consider spaces without magnetic fields and we recall some properties of them. In particular, we compare spaces, and the related norms, in the cases of whether the mass is present or not: these facts lead to different definitions of the norms, which however turn out to be equivalent. Here the modified Bessel function plays a fundamental role.
	
	As observed in \cite{FaFe}, for $m>0$, we can consider the usual space $$H^s(\R^3,\R) = \left\{u \in L^2(\R^3,\R) : |\xi|^s\cF u(\xi) \in L^2(\R^3,\R)\right\},$$ equipped with the norm
	\begin{equation}
		\label{real:Hs:norm}
		\begin{split}
			\|u\|_{H^s(\R^3,\R)}
			& = \left(\int_{\R^3} (|\xi|^{2s} + m^{2s}) |\cF u(\xi)|^2 \dd \xi\right)^{\frac12} \\
			& = \left(\frac{\wt{C}_s}{2}\int_{\R^3 \times \R^3} \frac{|u(x)-u(y)|^2}{|x-y|^{3+2s}} \dd x \dd y + m^{2s}\int_{\R^3}|u(x)|^2 \dd x \right)^{\frac12},
		\end{split}
	\end{equation}
	where $\wt{C}_s=s2^{2s}\Gamma\left(\frac{3+2s}{2}\right)/(\pi^{3/2}\Gamma(1-s))$, and the completion $H^s_m(\R^3,\R)$ of $C^{\infty}_c(\R^3,\R)$ with respect to the norm 
	\begin{equation}
		\label{real:Hsm:norm}
		\|u\|_{H^s_m(\R^3,\R)} := \left(\int_{\R^3} (|\xi|^2 + m^2)^s|\cF u(\xi)|^2 \dd \xi\right)^\frac{1}{2}.
	\end{equation}
	Since there exist $C_1,C_2>0$ such that 
	\begin{equation}
		\label{symbols:inequalities}
		C_1(|\xi|^{2s} + m^{2s}) \leq (|\xi|^2 + m^2)^s \leq C_2(|\xi|^{2s} + m^{2s}),
	\end{equation}
	we have that these are two ways to define the same space and the two norms \eqref{real:Hs:norm} and \eqref{real:Hsm:norm} are equivalent. Moreover, by \cite[Proposition 6]{FaFe},
	\begin{equation}
		\label{nomesemplice}
		\|u\|_{H^s_m(\R^3,\R)}
		= \left(\frac{C_s}{2}m^{\frac{3+2s}{2}}\int_{\R^3 \times \R^3} \frac{|u(x)-u(y)|^2}{|x-y|^{\frac{3+2s}{2}}}\mathcal{K}_{\frac{3+2s}{2}}(m|x-y|) \dd x \dd y + m^{2s}\int_{\R^3}|u(x)|^2 \dd x\right)^\frac{1}{2}.
	\end{equation}
	A similar argument can be repeated for
	$$H^s(\R^3,\C) = \left\{u \in L^2(\R^3,\C) : |\xi|^s\cF u(\xi) \in L^2(\R^3,\C)\right\},$$ equipped with the norm
	\begin{equation}
		\label{Hs:norm}
		\begin{split}
			\|u\|_{H^s(\R^3,\C)} &= \left(\int_{\R^3} (|\xi|^{2s} + m^{2s}) |\cF u(\xi)|^2 \dd \xi\right)^{\frac12} \\
			&= \left(\frac{\wt{C}_s}{2}\int_{\R^3 \times \R^3} \frac{|u(x)-u(y)|^2}{|x-y|^{3+2s}} \dd x \dd y + m^{2s}\int_{\R^3}|u(x)|^2 \dd x \right)^{\frac12}
		\end{split}
	\end{equation}
	and for $H^s_m(\R^3,\C)$, the completion of $C^{\infty}_c(\R^3,\C)$ with respect to the norm
	\begin{equation}
		\label{Hsm:norm}
		\|u\|_{H^s_m(\R^3,\C)} := \left(\int_{\R^3} (|\xi|^2 + m^2)^s|\cF{u}(\xi)|^2 \dd \xi\right)^\frac{1}{2}.
	\end{equation}
	Following the arguments of \cite[Proposition 6]{FaFe} we have
	\begin{Prop}
		\label{normHmsC}
		$H^s(\R^3,\C)=H_m^s(\R^3,\C)$, the respective norms \eqref{Hs:norm} and \eqref{Hsm:norm} are equivalent, and, for every $u \in H^s(\R^3,\C)$,
		\begin{equation}
			\label{Hs:equivalent:norm}
			\|u\|_{H^s_m(\R^3,\C)} = \left(\frac{C_s}{2}m^{\frac{3+2s}{2}}\int_{\R^3 \times \R^3} \frac{|u(x)-u(y)|^2}{|x-y|^{\frac{3+2s}{2}}}\mathcal{K}_{\frac{3+2s}{2}}(m|x-y|) \dd x \dd y + m^{2s}\int_{\R^3}|u(x)|^2 \dd x\right)^{\frac12}.
		\end{equation}
	\end{Prop}
	The proof of this Proposition is straightforward, and we report it in Appendix \ref{Appendix:Prop:6} for completeness.

	Also in the magnetic case, we could consider a {\em classical} fractional magnetic Sobolev space that does not involve the Bessel function (the one used in \cite{dASq}) equipped with an equivalent norm to 
	\begin{equation}
		\label{possible:equivalent:norm}
		\left(\int_{\R^3 \times \R^3} \frac{\left|\e^{-i(x-y)\cdot A\left(\frac{x+y}{2}\right)}u(x)-u(y) \right|^2}{|x-y|^{3+2s}} \dd x \dd y 
		+ 
		\int_{\R^3}|u|^2 \dd x\right)^\frac12.
	\end{equation}
	However, the equivalence between the norms \eqref{possible:equivalent:norm} and \eqref{magn:norm} seems to be not so straightforward, so we will work directly in the second space.\\
	Moreover, we observe that a similar equivalence as in Proposition \ref{normHmsC} can be repeated in the local setting. Indeed, if $\Omega$ is an open bounded subset of $\R^3$, consider the spaces $H^s(\Omega,\C)$ equipped with the norm
	\begin{equation}
		\label{HsKnorm}
		\|u\|_{H^s(\Omega,\C)}
		:=
		\left(\frac{C_s}{2}\int_{\Omega \times \Omega} \frac{|u(x)-u(y)|^2}{|x-y|^{3+2s}} \dd x \dd y
		+ m^{2s}\int_{\Omega}|u(x)|^2 \dd x \right)^{\frac12}
	\end{equation}
	and
	$H_m^s(\Omega,\C)$ equipped with the norm
	\begin{equation}
		\label{HmsKnorm}
		\|u\|_{H_m^s(\Omega,\C)}
		:=
		\left(\frac{C_s}{2}m^{\frac{3+2s}{2}}\int_{\Omega \times \Omega} \frac{|u(x)-u(y)|^2}{|x-y|^{\frac{3+2s}{2}}}\mathcal{K}_{\frac{3+2s}{2}}(m|x-y|) \dd x \dd y
		+ m^{2s}\int_{\Omega}|u(x)|^2 \dd x\right)^{\frac12}.
	\end{equation}
	We have
	\begin{Lem}
		\label{local:Hs:equivalence}
		The norms \eqref{HsKnorm} and \eqref{HmsKnorm} are equivalent and so $H^s(\Omega,\C)= H_m^s(\Omega,\C)$.
	\end{Lem}
	\begin{proof}
		Due to the continuity of $\mathcal{K}_\frac{3+2s}{2}$ for positive arguments, its positivity (see \eqref{bessel:positive}), and its asymptotic behavior in $0$ (see \eqref{Gauntest1}), there exist $C_1,C_2>0$, depending on $\Omega$, such that
		\begin{equation}
			\label{duestim}
			\frac{C_1}{|\zeta|^\frac{3+2s}{2}}
			\leq
			\mathcal{K}_\frac{3+2s}{2}(m|\zeta|)
			\leq
			\frac{C_2}{|\zeta|^\frac{3+2s}{2}}.
		\end{equation}
		and we get the statement immediately.
	\end{proof}

	\begin{Rem}\label{rem:dipR}
		Observe that the dependence on $\Omega$ of $C_1,C_2$ in the proof of Lemma \ref{local:Hs:equivalence} can be {\em reduced} to the dependence on $R>0$ such that $\Omega\subset B_{R}(\xi)$ for an arbitrary $\xi\in\mathbb{R}^3$ since, if $x,y\in\Omega$, $|x-y|<2R$.\footnote{Here we use only \eqref{Gauntest1}. Actually, by \eqref{Bessel:bound}, we can consider directly $C_2$ independent of $\Omega$.}
	\end{Rem}
	
	The following Lemma gives a local embedding result on the space $H^s$, which is well known if we consider the norm \eqref{HsKnorm} (see \cite[Proposition 2.5]{BrPa}). Due to the lack of homogeneity of the modified Bessel function $\mathcal{K}$, we cannot proceed directly but some adjustments are in order.
	
	\begin{Lem}
		\label{Lemma36}
		For every $\xi \in \R^3, R>0$, and $1 \leq q \leq 2^*_s$, there exists $C=C(s,m,R)>0$ such that for every $u \in H^s(B_R(\xi),\C)$,
		\[
		\|u\|_{L^q(B_R(\xi),\C)} \leq C\|u\|_{H^s_m(B_R(\xi),\C)}.
		\]
	\end{Lem}
	\begin{proof}
		Let $\xi \in \R^3$, $R>0$, $1 \leq q \leq 2^*_s$, and $u \in H^s(B_R(\xi),\C)$. Considering $H^s(B_R(\xi),\C)$ equipped with the norm \eqref{HsKnorm}, by \cite[Proposition 2.5]{BrPa}, there exists $C=C(s,m,R)>0$ (independent of $\xi$) such that
		\begin{equation}
			\label{crit:norm:estimate}
			\|u\|_{L^q(B_R(\xi),\C)} \leq C \|u\|_{H^s(B_R(\xi),\C)}.
		\end{equation}
		By Lemma \ref{local:Hs:equivalence} we have that the norms \eqref{HsKnorm} and \eqref{HmsKnorm} are equivalent, but, in general, the constant depends on $B_R(\xi)$. However, if $(x,y)\in B_R(\xi) \times B_R(\xi)$, then $|x-y|<2R$, so, arguing as in Remark \ref{rem:dipR}, we can consider constants $C_1, C_2$ independent of $\xi$ in the inequality \eqref{duestim}.
		Therefore, we obtain that in the right-hand side of \eqref{crit:norm:estimate} we can take the norm \eqref{HmsKnorm}, and so we can conclude.
	\end{proof}
	
	Finally, as in \cite[Lemma 2.3]{dASiSq} (see also \cite[Lemma 2.2]{dASiSq}) and using, eventually, the equivalence in Proposition \ref{normHmsC}, we can show the following tool that will be useful in our concentration-compactness arguments.
	\begin{Lem}
		\label{vanishing:lemma}
		Let $\{u_n\}_n \subset H^s(\R^3,\C)$ be a bounded sequence. Assume that for some $R>0$ and $q \in [2,2^*_s)$
		\[
		\lim_n \sup_{\xi \in \R^3}\int_{B_R(\xi)}|u_n(x)|^q \dd x = 0.
		\]
		Then $u_n \to 0$ in $L^p(\R^3,\C)$ for $p \in (2,2^*_s)$.
	\end{Lem}

	\subsection{Properties of \texorpdfstring{$H^s_A(\R^3,\C)$}{HsA(R3,C)}}
	
	Now, we explore some properties, in particular embedding results, of $H^s_A(\R^3,\C)$.
	
	As usual in the magnetic Sobolev spaces, the diamagnetic inequality is a fundamental tool. In our functional setting, due to the {\em pointwise diamagnetic inequality}
	\begin{equation}
		\label{pointwise:diamagnetic}
		||u(x)|-|u(y)||
		\leq 
		\left|\e^{-i(x-y)\cdot A\left(\frac{x+y}{2}\right)}u(x)-u(y)\right|
		\text{ for a.e. } x,y \in \R^3,
	\end{equation}
	see \cite[Remark 3.2]{dASq}, we easily have
	\begin{Lem}[Diamagnetic inequality]
		\label{diamagnetic:eq}
		For every $u \in H^s_A(\R^3,\C)$ it holds $\||u|\|_{H^s_m(\R^3,\R)} \leq \|u\|_{A,s}$ and so $|u| \in H^s(\R^3,\R)$.
	\end{Lem}
	
	In the next Lemmas we show some embedding results.
	
	Let $H_{\rm loc}^s(\R^3,\C):=\set{u:\R^3\to\C : u\in H^s(\Omega,\C) \text{ for all open } \Omega \subset\subset \R^3}$.
	\begin{Lem}[Local embedding]
		\label{local:cont:emb}
		$H^s_A(\R^3,\C) \hookrightarrow H^s_{\loc}(\R^3,\C)$.
	\end{Lem}
	\begin{proof}
		Let $\Omega$ be an open subset of $\R^3$ such that $\Omega \subset\subset \R^3$. Then,
		\begin{align*}
			\|u\|_{H_m^s(\Omega,\C)}^2 
			&\leq m^{2s}\int_{\Omega}|u(x)|^2 \dd x
			+ C_sm^{\frac{3+2s}{2}}\int_{\Omega \times \Omega}\frac{\left|\e^{-i(x-y)\cdot A\left(\frac{x+y}{2}\right)}u(x) - u(y)\right|^2}{|x-y|^{\frac{3+2s}{2}}}\mathcal{K}_{\frac{3+2s}{2}}(m|x-y|) \dd x \dd y \\
			&\qquad + C_sm^{\frac{3+2s}{2}}\int_{\Omega \times \Omega}\frac{\left|u(x)\left(\e^{-i(x-y)\cdot A\left(\frac{x+y}{2}\right)}-1\right)\right|^2}{|x-y|^{\frac{3+2s}{2}}}\mathcal{K}_{\frac{3+2s}{2}}(m|x-y|) \dd x \dd y \\
			& \leq 2\|u\|_{A,s}^2 + C_sm^{\frac{3+2s}{2}}\int_{\Omega \times \Omega}\frac{\left|u(x)\left(\e^{-i(x-y)\cdot A\left(\frac{x+y}{2}\right)}-1\right)\right|^2}{|x-y|^{\frac{3+2s}{2}}}\mathcal{K}_{\frac{3+2s}{2}}(m|x-y|) \dd x \dd y.
		\end{align*}
		Moreover, since, in general, $|\e^{i\theta} - 1| \leq 2$ and, for small $\theta \in \R$, $|\e^{i\theta} - 1| \leq C|\theta|$, we have, by \eqref{Bessel:bound},
		\begin{align*}
			&
			\int_{\Omega \times \Omega}\frac{\left|u(x)\left(\e^{-i(x-y)\cdot A\left(\frac{x+y}{2}\right)}-1\right)\right|^2}{|x-y|^{\frac{3+2s}{2}}}\mathcal{K}_{\frac{3+2s}{2}}(m|x-y|) \dd x \dd y \\
			&  \quad =\int_{\Omega}|u(x)|^2 \dd x \int_{\Omega \cap \left\{|x-y| \geq 1\right\}}\frac{\left|\e^{-i(x-y)\cdot A\left(\frac{x+y}{2}\right)}-1\right|^2}{|x-y|^{\frac{3+2s}{2}}}\mathcal{K}_{\frac{3+2s}{2}}(m|x-y|) \dd y \\
			& \quad\qquad + \int_{\Omega}|u(x)|^2 \dd x \int_{\Omega \cap \left\{|x-y| < 1\right\}}\frac{\left|\e^{-i(x-y)\cdot A\left(\frac{x+y}{2}\right)}-1\right|^2}{|x-y|^{\frac{3+2s}{2}}}\mathcal{K}_{\frac{3+2s}{2}}(m|x-y|) \dd y\\
			& \quad\leq C\left(\int_{\Omega}|u(x)|^2 \dd x \int_{\Omega \cap \left\{|x-y| \geq 1\right\}}\frac{1}{|x-y|^{3+2s}} \dd y \right.\\
			& \quad\qquad \left.+ \|A\|^2_{L^{\infty}(\Omega)}\int_{\Omega}|u(x)|^2 \dd x \int_{K \cap \left\{|x-y| < 1\right\}}\frac{1}{|x-y|^{1+2s}} \dd y \right)\\
			& \quad\leq C(s,m,\|A\|_{L^{\infty}(\Omega)})\int_{\Omega}|u(x)|^2 \dd x \left(\int_1^{+\infty}\frac{1}{r^{1+2s}} \dd r + \int_0^1\frac{1}{r^{2s-1}} \dd r \right)\\
			& \quad\leq C(s,m,\|A\|_{L^{\infty}(\Omega)})\int_{\Omega}|u(x)|^2 \dd x.
		\end{align*}
		Therefore we can conclude.
	\end{proof}
	
	\begin{Lem}
		\label{precompact:Lq}
		Let $\{A_n\}_n$ be a sequence of functions $A_n :\R^3 \to \R^3$ uniformly locally bounded and with locally bounded gradient and $\{u_n\}_n$ be a sequence such that, for every $n \in \N$, $u_n \in H^s_{A_n}(\R^3,\C)$ and
		\begin{equation}
			\label{sup:un:bounded}
			\sup_{n \in \N} \|u_n\|_{A_n,s} < +\infty.
		\end{equation}
		Then, up to a subsequence, $\{u_n\}_n$ converges strongly in $L^q_{\loc}(\R^3,\C)$, for every $q \in [1,2^*_s)$.
	\end{Lem}
	\begin{proof}
		Let $\Omega$ be an open subset of $\R^3$ such that $\Omega \subset\subset \R^3$. Arguing as in Lemma \ref{local:cont:emb} we have that, for every $n \in \N$,
		\[
		\|u_n\|_{H_m^s(\Omega,\C)}^2
		\leq
		\|u_n\|_{A_n,s}^2
		+ \frac{C_s}{2}m^{\frac{3+2s}{2}}\int_{\Omega \times \Omega}\frac{\left|u_n(x)\left(\e^{-i(x-y)\cdot A_n\left(\frac{x+y}{2}\right)}-1\right)\right|^2}{|x-y|^{\frac{3+2s}{2}}}\mathcal{K}_{\frac{3+2s}{2}}(m|x-y|) \dd x \dd y
		\]
		and, using the uniform locally boundedness of $\{A_n\}_n$,
		\[
		\int_{\Omega \times \Omega}\frac{\left|u_n(x)\left(\e^{-i(x-y)\cdot A_n\left(\frac{x+y}{2}\right)}-1\right)\right|^2}{|x-y|^{\frac{3+2s}{2}}}\mathcal{K}_{\frac{3+2s}{2}}(m|x-y|) \dd x \dd y
		\leq
		C(s,m)\int_{\Omega}|u_n(x)|^2 \dd x,
		\]
		so that, $\|u_n\|_{H_m^s(\Omega,\C)}\leq C \|u_n\|_{A_n,s}$, and then, by \eqref{sup:un:bounded}, $\{\|u_n\|_{H_m^s(\Omega)}\}_n$ is bounded.\\
		Hence, we can conclude thanks to Lemma \ref{local:Hs:equivalence}, and using \cite[Corollary 7.2]{DiPaVa}.
	\end{proof}
	
	\begin{Lem}[Magnetic Sobolev embedding]
		\label{magn:sob:emb}
		If $q \in [2,2^*_s]$, the embedding $H^s_A(\R^3,\C) \hookrightarrow L^q(\R^3,\C)$ is continuous and, if $q \in [1,2^*_s)$, $H^s_A(\R^3,\C) \hookrightarrow L^q_{\loc}(\R^3,\C)$ is compact.
	\end{Lem}
	\begin{proof}
		Let $u\in H^s_A(\R^3,\C)$. By Lemma \ref{diamagnetic:eq}, $|u|\in H^s(\R^3,\R)$ and so, by \cite[Theorem 6.5]{DiPaVa}, $|u|\in L^q(\R^3,\R)$, for $q \in [2,2^*_s]$.
		Moreover, since, as observed before $H^s(\R^3,\R)=H_m^s(\R^3,\R)$, due to the equivalence of the respective norms, and by \eqref{nomesemplice}, using also \eqref{pointwise:diamagnetic}, we get
		\[
		\|u\|_{L^{2^*_s}(\R^3,\C)}^2
		=
		\||u|\|_{L^{2^*_s}(\R^3,\R)}^2
		\leq
		C \| |u|\|_{H_m^s(\R^3,\R)}^2
		\leq
		C\|u\|^2_{A,s}.
		\]
		Finally, the compact embedding follows by Lemma \ref{local:cont:emb} and \cite[Corollary 7.2]{DiPaVa}.
	\end{proof}
	Arguing as in \cite[Lemma 3.8]{dASq} (see also \cite[Lemma 5.3]{DiPaVa}) we have the following cut-off estimates.
	\begin{Lem}
		\label{cut-off}
		Let $\varphi \in C^{0,1}(\R^3,\R)$ such that $0 \leq \varphi \leq 1$. There exists $C>0$ such that for every $u \in H^s_A(\R^3,\C)$ and every pair of measurable sets $E_1,E_2 \subset \R^3$, it holds
		\begin{multline*}
			\int_{E_1 \times E_2} \frac{|\e^{-i(x-y)\cdot A\left(\frac{x+y}{2}\right)}\varphi(x)u(x)-\varphi(y)u(y)|^2}{|x-y|^{\frac{3+2s}{2}}}\mathcal{K}_{{\frac{3+2s}{2}}}(m|x-y|) \dd x \dd y \\
			\leq C\left(\min\left\{\int_{E_1}|u|^2 \dd x, \int_{E_2}|u|^2 \dd x \right\} + \int_{E_1 \times E_2} \frac{|\e^{-i(x-y)\cdot A\left(\frac{x+y}{2}\right)}u(x)-u(y)|^2}{|x-y|^{\frac{3+2s}{2}}}\mathcal{K}_{{\frac{3+2s}{2}}}(m|x-y|) \dd x \dd y\right),
		\end{multline*}
		where $C$ depends on $s$ and on the Lipschitz constant of $\vp$.
	\end{Lem}
	\begin{proof}
		By \eqref{Bessel:bound},
		\begin{align*}
			&\int_{E_1 \times E_2} \frac{|\e^{-i(x-y)\cdot A\left(\frac{x+y}{2}\right)}\varphi(x)u(x)-\varphi(y)u(y)|^2}{|x-y|^{\frac{3+2s}{2}}}\mathcal{K}_{{\frac{3+2s}{2}}}(m|x-y|) \dd x \dd y \\
			& \leq 2\int_{E_1 \times E_2} \frac{|\e^{-i(x-y)\cdot A\left(\frac{x+y}{2}\right)}\varphi(x)u(x) - \varphi(x)u(y) |^2}{|x-y|^{\frac{3+2s}{2}}}\mathcal{K}_{{\frac{3+2s}{2}}}(m|x-y|) \dd x \dd y \\
			& \qquad + 2\int_{E_1 \times E_2} \frac{|\varphi(x)u(y)-\varphi(y)u(y) |^2}{|x-y|^{\frac{3+2s}{2}}}\mathcal{K}_{{\frac{3+2s}{2}}}(m|x-y|) \dd x \dd y \\
			& \leq 2\int_{E_1 \times E_2} \frac{|\e^{-i(x-y)\cdot A\left(\frac{x+y}{2}\right)}u(x) - u(y) |^2}{|x-y|^{\frac{3+2s}{2}}}\mathcal{K}_{{\frac{3+2s}{2}}}(m|x-y|) \dd x \dd y + C\int_{E_1 \times E_2} \frac{|u(y)|^2|\vp(x)-\vp(y)|^2}{|x-y|^{3+2s}} \dd x \dd y
		\end{align*}
		and, since $\vp \in C^{0,1}(\R^3,\R)$ and $0 \leq \vp \leq 1$,
		\begin{align*}
			&\int_{E_1 \times E_2} \frac{|u(y)|^2|\vp(x)-\vp(y)|^2}{|x-y|^{3+2s}} \dd x \dd y \\
			& \leq C\int_{E_2}|u(y)|^2 \int_{E_1 \cap \{|x-y| \leq 1\}}\frac{1}{|x-y|^{1+2s}} \dd x \dd y + 4\int_{E_2}|u(y)|^2 \int_{E_1 \cap \{|x-y| > 1\}}\frac{1}{|x-y|^{3+2s}} \dd x \dd y \\
			& \leq C\left(\int_{E_2}|u(y)|^2 \int_0^1r^{1-2s} \dd r \dd y + \int_{E_2}|u(y)|^2 \int_1^{+\infty}\frac{1}{r^{1+2s}} \dd r \dd y \right) \leq C\int_{E_2}|u(y)|^2 \dd y.
		\end{align*}
		On the other hand, 
		\begin{align*}
			&\int_{E_1 \times E_2} \frac{|\e^{-i(x-y)\cdot A\left(\frac{x+y}{2}\right)}\varphi(x)u(x)-\varphi(y)u(y)|^2}{|x-y|^{\frac{3+2s}{2}}}\mathcal{K}_{{\frac{3+2s}{2}}}(m|x-y|) \dd x \dd y \\
			& \leq
			2 \int_{E_1 \times E_2} \frac{|\e^{-i(x-y)\cdot A\left(\frac{x+y}{2}\right)}\varphi(x)u(x)-\e^{-i(x-y)\cdot A\left(\frac{x+y}{2}\right)}\varphi(y)u(x)|^2}{|x-y|^{\frac{3+2s}{2}}}\mathcal{K}_{{\frac{3+2s}{2}}}(m|x-y|) \dd x \dd y \\
			& \quad
			+ 2 \int_{E_1 \times E_2} \frac{|\e^{-i(x-y)\cdot A\left(\frac{x+y}{2}\right)}\varphi(y)u(x)-\varphi(y)u(y)|^2}{|x-y|^{\frac{3+2s}{2}}}\mathcal{K}_{{\frac{3+2s}{2}}}(m|x-y|) \dd x \dd y \\
			& \leq 2 \int_{E_1 \times E_2} \frac{|\e^{-i(x-y)\cdot A\left(\frac{x+y}{2}\right)}u(x) - u(y) |^2}{|x-y|^{\frac{3+2s}{2}}}\mathcal{K}_{{\frac{3+2s}{2}}}(m|x-y|) \dd x \dd y + C \int_{E_1 \times E_2} \frac{|u(x)|^2|\varphi(x)-\varphi(y)|^2}{|x-y|^{3+2s}}\dd x \dd y
		\end{align*}
		and, arguing as before, we conclude.
	\end{proof}
	Moreover, arguing as in \cite[Lemma 3.10 and Lemma 3.11]{dASq}, we get the following partial Gauge invariance properties.
	\begin{Lem}
		\label{partial:gauge}
		Let $\xi, \eta \in \R^3$, and $u \in H^s_A(\R^3,\C)$. If 
		\[
		v(x):=\e^{i\eta\cdot x}u(x+\xi), \quad A_{\eta}(x):=A(x+\xi) + \eta, \quad x \in \R^3,
		\]
		then, $v \in H^s_{A_{\eta}}(\R^3,\C)$ and $[u]_{A,s} = [v]_{A_\eta,s}$.\\ 
		In particular, if $A$ is linear and $\eta=-A(\xi)$, then $[u]_{A,s} = [v]_{A,s}$.
	\end{Lem}
	
	Finally, if we consider
	\[
	H^s_{A,\rad}(\R^3,\C):=\{u \in H^s_A(\R^3,\C) : u \text{ radially symmetric}\},
	\]
	as usual in such a symmetric setting, a fundamental tool is a compact embedding. In our case we have
	\begin{Lem}
		\label{radial:comp:emb}
		For every $q \in (2,2^*_s)$, the embedding $u\in H^s_{A,\rad}(\R^3,\C) \longmapsto |u| \in L^q(\R^3,\R)$, is compact.
	\end{Lem}
	\begin{proof}
		Since, by Lemma \ref{diamagnetic:eq}, the embedding $u \in H^s_A(\R^3,\C) \longmapsto |u|\in H^s(\R^3,\R)$ is continuous, we apply \cite[Theorem II.1]{Lions1982} and we can conclude.
	\end{proof}
	
	\section{Behavior as \texorpdfstring{$s \nearrow 1$}{s -> 1}}
	\label{section:limit:behavior}
	In this section, we study the behavior of our operator \eqref{pseudo:mag:Sch:order:s} as $s \nearrow 1$.\\
	Our aim is twofold. From one hand we get the convergence results \`a la Bourgain-Brezis-Mironescu for open and  bounded $\Omega \subset \R^3$ with Lipschitz boundary and for $\R^3$.
	On the other hand we show how to {\em remove the singularity} from the definition of \eqref{pseudo:mag:Sch:order:s}.
	
	\subsection{Bourgain-Brezis-Mironescu type formulas}
	Let $\Omega \subset \R^3$ be an open bounded set with Lipschitz boundary and define $H^1_A(\Omega,\C)$ as the space of functions $u \in L^2(\Omega,\C)$ such that 
	\[
	\int_{\Omega} \left|\nabla_A u \right|^2 \dd x < +\infty,
	\]
	equipped with the norm
	\[
	\|u\|_{H^1_A(\Omega,\C)}:=\left(\int_{\Omega} \left|\nabla_A u\right|^2 \dd x + m^2\int_{\Omega} u^2 \dd x \right)^{\frac12}.
	\]
	The proof of Theorem \ref{main:result:BBM:bounded} relies on the following result due to Squassina and Volzone (see \cite[Theorem 2.5]{SqVo}): for the readers' convenience, we report the statement here.
	\begin{Th}
		\label{Th:2.5:SqVo}
		Assume $A \in C^2(\overline{\Omega},\C)$ and let $u \in H^1_A(\Omega,\C)$. Consider a sequence $\{\rho_n\}_n$ of nonnegative functions defined in $(0,+\infty)$, such that 
		\[
		\int_0^{+\infty}\rho_n(r)r^2 \dd r <+\infty \text{ for every } n \in \N,
		\]
		\begin{equation*}
			\lim_n\int_0^{+\infty}\rho_n(r)r^2 \dd r = 1,
		\end{equation*}
		and, for every $\delta>0$,
		\begin{equation}
			\label{rho:hyp:2}
			\lim_n\int_{\delta}^{+\infty}\rho_n(r)r^2 \dd r = 0.	
		\end{equation}
		Then, we have 
		\begin{equation*}
			\lim_n\int_{\Omega \times \Omega} \frac{|u(x)-\e^{i(x-y)\cdot A\left(\frac{x+y}{2}\right)}u(y)|^2}{|x-y|^2} \rho_n(|x-y|) \dd x \dd y = \frac{4\pi}{3}\int_{\Omega}|\nabla_A u|^2 \dd x.
		\end{equation*}
	\end{Th}
	Therefore we can proceed as follows.
	\begin{proof}[Proof of Theorem \ref{main:result:BBM:bounded}]
		In order to apply Theorem \ref{Th:2.5:SqVo}, we need to find a suitable sequence $\{\rho_n\}_n$.\\ 
		Let $r_{\Omega}:=\diam(\Omega)$ and let $\psi \in C^{\infty}([0,+\infty),[0,1])$ be a cut-off function such that 
		\[
		\psi(r) =
		\begin{cases}
			1, &0 \leq r < r_{\Omega},\\
			0, &r > 2r_{\Omega}.
		\end{cases}
		\]
		It follows that $\psi(|x-y|)=1$ for every $x,y \in \Omega$. Now, let $\{s_n\}_n \subset (0,1)$ be a sequence such that $s_n \nearrow 1$ as $n \to +\infty$ and consider the sequence
		\begin{equation*}
			\rho_n(r):=\frac{2\pi}{3}C_{s_n}m^{\frac{3+2s_n}{2}}r^{\frac{1-2s_n}{2}}\bessel{\frac{3+2s_n}{2}}{mr}\psi(r), \qquad r > 0. 
		\end{equation*}
		By \eqref{Bessel:bound}, since $\{s_n\}_n \in (0,1)$, we have
		\[
		\int_0^{+\infty} \rho_n(r)r^2 \dd r \leq C\frac{s_n2^{\frac{2s_n+1}{2}}}{3\sqrt{\pi}\Gamma(1-s_n)}\int_0^{2r_{\Omega}} r^{1-2s_n} \dd r < +\infty,
		\]
		obtaining the required summability.\\
		With the same computations, we can get \eqref{rho:hyp:2}. Indeed, using the monotonicity of $\mathcal{K}_\nu$ with respect to $\nu$ and \eqref{Bessel:bound}, for every $\delta > 0$ we have
		\begin{align*}
			\int_{\delta}^{+\infty} \rho_n(r)r^2 \dd r &\leq \frac{2\pi}{3}C_{s_n}m^{\frac{3+2s_n}{2}}\int_{\delta}^{2r_{\Omega}}r^{\frac{1-2s_n}{2}}\bessel{\frac{5}{2}}{mr}r^2 \dd r \leq C\frac{s_n2^{\frac{2s_n+1}{2}}m^{s_n-1}}{3\sqrt{\pi}\Gamma(1-s_n)}\int_{\delta}^{2r_{\Omega}} r^{-s_n} \dd r \\
			& = C\frac{s_n2^{\frac{2s_n+1}{2}}m^{s_n-1}}{3\sqrt{\pi}\Gamma(2-s_n)}\left[(2r_{\Omega})^{1-s_n} - \delta^{1-s_n}\right] \to 0 \text{ as } n \to +\infty.
		\end{align*}
		To conclude, let us write 
		\[
		\int_0^{+\infty} \rho_n(r)r^2 \dd r = I_{1,n} + I_{2,n},
		\]
		where
		\[
		I_{1,n} = \frac{2\pi}{3}C_{s_n}m^{\frac{3+2s_n}{2}}\int_0^{+\infty}r^{\frac{5-2s_n}{2}}\bessel{\frac{3+2s_n}{2}}{mr} \dd r
		\]
		and
		\[
		I_{2,n} = \frac{2\pi}{3}C_{s_n}m^{\frac{3+2s_n}{2}}\int_0^{+\infty}r^{\frac{5-2s_n}{2}}\bessel{\frac{3+2s_n}{2}}{mr}(\psi(r)-1) \dd r.
		\]
		By \eqref{integral:repr:Bessel} and applying the change of variable $x=(mr)^2/4t$, we get
		\begin{equation}
			\label{I1:integral}
			\begin{aligned}
				I_{1,n} &= \frac{2\pi}{3}C_{s_n}\frac{m^{3+2s_n}}{2^{\frac{5+2s_n}{2}}}\int_0^{+\infty}\e^{-t}t^{-\frac{5+2s_n}{2}}\int_0^{+\infty}r^4\e^{-\frac{m^2r^2}{4t}} \dd r \dd t\\
				& = \frac{2\pi}{3}C_{s_n}\frac{m^{2s_n-2}}{2^{\frac{2s_n-3}{2}}}\left(\int_0^{+\infty}\e^{-t}t^{-s_n}\dd t\right)\left(\int_0^{+\infty}x^{\frac32}\e^{-x} \dd x\right) \\
				& = \frac{2\pi}{3}C_{s_n}\frac{m^{2s_n-2}}{2^{\frac{2s_n-3}{2}}}\Gamma(1-s_n)\Gamma\left(\frac52\right) = s_nm^{2s_n-2} \to 1 \text{ as } n \to +\infty.
			\end{aligned}
		\end{equation}
		Moreover, using the monotonicity of $\mathcal{K}_\nu$ with respext to $\nu$ and \eqref{Gauntest2}, we have
		\begin{equation}
			\label{I2:integral}
			\begin{aligned}
				|I_{2,n}| &\leq \frac{2\pi}{3}C_{s_n}m^{\frac{3+2s_n}{2}}\int_{r_{\Omega}}^{+\infty} r^{\frac{5-2s_n}{2}}\bessel{\frac{3+2s_n}{2}}{mr} \dd r \\    
				& \leq \frac{2\pi}{3}C_{s_n}m^{\frac{3+2s_n}{2}}\int_{r_{\Omega}}^{+\infty} r^{\frac{5-2s_n}{2}}\bessel{\frac52}{mr} \dd r \\
				& \leq C\frac{s_n2^{\frac{2s_n-1}{2}}}{\Gamma(1-s_n)}m^{1+s_n} \int_{r_{\Omega}}^{+\infty}\frac{r^2}{r^{s_n}}\e^{-mr} \dd r \\
				& \leq C\frac{s_n2^{\frac{2s_n-1}{2}}}{r_{\Omega}^{s_n}\Gamma(1-s_n)}m^{1+s_n} \int_{r_{\Omega}}^{+\infty} r^2\e^{-mr} \dd r \to 0, \text{ as } n \to +\infty.
			\end{aligned}
		\end{equation}
		Therefore, applying Theorem \ref{Th:2.5:SqVo}, we get
		\[
		\int_{\Omega \times \Omega} \frac{\left|u(x)-\e^{i(x-y)\cdot A\left(\frac{x+y}{2}\right)}u(y) \right|^2}{|x-y|^2}\rho_n(|x-y|) \dd x \dd y
		\to
		\frac{4\pi}{3}\int_{\Omega}|\nabla_Au|^2 \dd x.
		\]
		Hence,
		\begin{align*}
			&\frac{C_{s_n}}{2}  m^{\frac{3+2s_n}{2}}\int_{\Omega \times \Omega} \frac{\left|u(x)-\e^{i(x-y)\cdot A\left(\frac{x+y}{2}\right)}u(y) \right|^2}{|x-y|^\frac{3+2s_n}{2}}\bessel{\frac{3+2s_n}{2}}{m|x-y|} \dd x \dd y \\
			&\qquad = \frac{3}{4\pi}\int_{\Omega \times \Omega} \frac{\left|u(x)-\e^{i(x-y)\cdot A\left(\frac{x+y}{2}\right)}u(y) \right|^2}{|x-y|^2}\rho_n(|x-y|) \dd x \dd y 
			\to
			\int_{\Omega}|\nabla_Au|^2 \dd x,
		\end{align*}
		as $n \to +\infty$, concluding the proof.
	\end{proof}
	Now, let us move to the whole space setting. The proof strategy is almost the same as before (it relies on using a suitable sequence of mollifiers $\{\rho_n\}_n$). However, a slightly different assumption on the magnetic field must be considered. We are going to apply the following result (\cite[Theorem 1.1 and Remark 2.1]{NgPiSqVe}).
	\begin{Th}
		\label{Th:1.1:NgPiSqVe}
		Let $A:\R^3 \to \R^3$ be Lipschitz and let $\{\rho_n\}_n$ be a sequence of nonnegative functions defined in $(0,+\infty)$, such that 
		\[
		\int_0^{+\infty}\rho_n(r)r^2 \dd r <+\infty \text{ for every } n \in \N,
		\]
		\begin{equation}
			\label{rho:hyp:1:bis}
			\int_0^{+\infty}\rho_n(r)r^2 \dd r = 1,
		\end{equation}
		and, for every $\delta>0$,
		\begin{equation}
			\label{rho:hyp:2:bis}
			\lim_n\int_{\delta}^{+\infty}\rho_n(r)r^2 \dd r = 0.	
		\end{equation}
		Then, for $u \in H^1_A(\R^3,\C)$, we have
		\[
		\lim_n \int_{\R^3 \times \R^3} \frac{|u(x)-\e^{i(x-y)\cdot A\left(\frac{x+y}{2}\right)}u(y)|^2}{|x-y|^2} \rho_n(|x-y|) \dd x \dd y = \frac{4\pi}{3}\int_{\R^3}|\nabla_A u|^2 \dd x.
		\]
	\end{Th}
	Thus we get
	\begin{proof}[Proof of Theorem \ref{main:result:BBM:unbounded}]
		Consider a sequence $\{s_n\}_n \subset (0,1)$ such that $s_n \nearrow 1$ as $n \to +\infty$ and let
		\begin{equation*}
			\rho_n(r):=\frac{2\pi}{3}\frac{C_{s_n}}{s_n}m^{\frac{7-2s_n}{2}}r^{\frac{1-2s_n}{2}}\bessel{\frac{3+2s_n}{2}}{mr}, \qquad r > 0.
		\end{equation*}
		We show that this sequence satisfies the hypotheses of Theorem \ref{Th:1.1:NgPiSqVe}. \\
		Fix $\delta > 0$. By the monotonicity of $\mathcal{K}_\nu$ with respect to $\nu$ and \eqref{Gauntest2}, we get
		\begin{align*}
			\int_{\delta}^{+\infty} \rho_n(r)r^2 \dd r &\leq C\frac{C_{s_n}}{s_n}m^{3-s_n}\int_{\delta}^{+\infty} r^{2-s_n}\e^{-mr} \dd r \to 0 \text{ as } n \to +\infty.
		\end{align*}
		Hence, \eqref{rho:hyp:2:bis} holds. Moreover, arguing as in \eqref{I1:integral},
		\begin{align*}
			\int_0^{+\infty} \rho_n(r)r^2 \dd r = 1.
		\end{align*}
		and \eqref{rho:hyp:1:bis} follows.\\
		Therefore, by Theorem \ref{Th:1.1:NgPiSqVe}, we have
		\[
		\int_{\R^3 \times \R^3} \frac{|u(x)-\e^{i(x-y)\cdot A\left(\frac{x+y}{2}\right)}u(y)|^2}{|x-y|^2} \rho_n(|x-y|) \dd x \dd y \to \frac{4\pi}{3}\int_{\R^3}|\nabla_Au|^2 \dd x.
		\]
		Hence,
		\begin{align*}
			&\frac{C_{s_n}}{2}m^{\frac{3+2s_n}{2}}\int_{\R^3 \times \R^3} \frac{\left|u(x)-\e^{i(x-y)\cdot A\left(\frac{x+y}{2}\right)}u(y) \right|^2}{|x-y|^\frac{3+2s_n}{2}}\bessel{\frac{3+2s_n}{2}}{m|x-y|} \dd x \dd y \\
			&\qquad = \frac{3}{4\pi}s_nm^{2s_n-2}\int_{\R^3 \times \R^3} \frac{\left|u(x)-\e^{i(x-y)\cdot A\left(\frac{x+y}{2}\right)}u(y) \right|^2}{|x-y|^2}\rho_n(|x-y|) \dd x \dd y 
			\to \int_{\R^3}|\nabla_Au|^2 \dd x
		\end{align*}
		as $n \to +\infty$.
	\end{proof}

	\subsection{Convergence of the operator}
	As for the nonmagnetic case, it is possible to remove the singularity from the integral definition of \eqref{pseudo:mag:Sch:order:s}. To do this, we need to write our operator in a different integral form (see \cite[Theorem 3.2]{DiPaVa} for the classical fractional Laplacian and \cite[Theorem 4.5.2]{Ambrosio2021-1} for the pseudorelativistic fractional case). We start by stating two preliminary results contained in \cite{dASq}.
	\begin{Lem}[\cite{dASq}, Lemma 2.4]
		\label{Lemma_2.4}
		Let $K$ be a compact subset of $\R^3$, $R>0$ and set $K':=\{x \in \R^3 : d(x,K) \leq R\}$. Assume that $f \in C^2(\R^3 \times \R^3,\C)$ and that $g \in C^{1,\gamma}(K',\C)$ for some $\gamma \in [0,1]$. If $h(x,y)=f(x,y)g(y)$, then there exists a positive constant $C$ depending on $K,f,g,R$, such that
		\[
		\left|\nabla_y h(x,y_2) - \nabla_y h(x,y_1)\right| \leq C|y_2-y_1|^{\gamma},
		\]
		for all $x \in K$ and every $y_1,y_2 \in K'$.
	\end{Lem}
	A consequence of the previous Lemma is 
	\begin{Lem}[\cite{dASq}, Lemma 2.5]
		\label{numerator:control:origin}
		Let $A \in C^2(\R^3,\C)$ and $u \in C^{1,\gamma}_{\loc}(\R^3,\C)$ for some $\gamma \in [0,1]$. Then, for any compact set $K \subset \R^3$ and $R>0$, there exists a positive constant $C$ depending on $R,K,A,u$, such that
		\[
		\left|\e^{-iy\cdot A\left(x +\frac{y}{2}\right)}u(x+y) + \e^{iy\cdot A\left(x -\frac{y}{2}\right)}u(x-y) - 2u(x)\right| \leq C|y|^{1+\gamma},
		\]
		for every $x \in K$ and $y \in B_R(0)$.
	\end{Lem}
	Now, for $A$ and $u$ smooth enough, we show an alternative definition of the operator defined in \eqref{pseudo:mag:Sch:order:s}, which allows for better handling of computation, useful in the proof of the main Theorem of this subsection.
	\begin{Prop}
		\label{remove:singularity}
		Let $A \in C^2(\R^3,\C)$, $m>0$ and $u \in L^{\infty}(\R^3,\C) \cap C^{1,\gamma}_{\loc}(\R^3,\C)$, for some $\gamma\in (0,1]$ with $\gamma > 2s-1$. Then, for a.e. $x\in\R^3$,
		\begin{multline*}
			(-\Delta + m^2)_A^su(x)\\
			=-\frac{C_s}{2}m^{\frac{3+2s}{2}}\int_{\R^3}\frac{\e^{-iz\cdot A\left(x +\frac{z}{2}\right)}u(x+z) + \e^{iz\cdot A\left(x -\frac{z}{2}\right)}u(x-z) - 2u(x)}{|z|^{\frac{3+2s}{2}}}\bessel{\frac{3+2s}{2}}{m|z|} \dd z + m^{2s}u(x).
		\end{multline*}
	\end{Prop}
	\begin{proof}
		For every $\eps>0$ we consider the \textit{auxiliary} function $g_{\eps}:\R^3 \to \R$ defined as
		\begin{equation}
			\label{auxiliary:function}
			g_{\eps}(x):=\int_{\R^3}\frac{u(x)-\e^{i(x-y)\cdot A\left(\frac{x+y}{2}\right)}u(y)}{|x-y|^{\frac{3+2s}{2}}}\bessel{\frac{3+2s}{2}}{m|x-y|} 1_{B^c_{\eps}(x)}(y) \dd y
		\end{equation}
		Performing the two changes of variable $z=y-x$ and $z=x-y$ we obtain
		\begin{equation}
			\label{first:cov}
			g_{\eps}(x) = \int_{\R^3}\frac{u(x)-\e^{-iz\cdot A\left(x+\frac{z}{2}\right)}u(x+z)}{|z|^{\frac{3+2s}{2}}}\bessel{\frac{3+2s}{2}}{m|z|} 1_{B^c_{\eps}(0)}(z) \dd z 
		\end{equation}
		and
		\begin{equation}
			\label{second:cov}
			g_{\eps}(x) = \int_{\R^3}\frac{u(x)-\e^{iz\cdot A\left(x-\frac{z}{2}\right)}u(x-z)}{|z|^{\frac{3+2s}{2}}}\bessel{\frac{3+2s}{2}}{m|z|} 1_{B^c_{\eps}(0)}(z) \dd z.
		\end{equation}
		Putting together \eqref{first:cov} and \eqref{second:cov}, we get
		\[
		g_{\eps}(x) = -\frac12\int_{\R^3}\frac{\e^{-iz\cdot A\left(x+\frac{z}{2}\right)}u(x+z) + \e^{iz\cdot A\left(x-\frac{z}{2}\right)}u(x-z)- 2u(x)}{|z|^{\frac{3+2s}{2}}}\bessel{\frac{3+2s}{2}}{m|z|} 1_{B^c_{\eps}(0)}(z)\dd z.
		\]
		Now, for $x \in \R^3$,
		\begin{multline*}
			\frac{\e^{-iz\cdot A\left(x+\frac{z}{2}\right)}u(x+z) + \e^{iz\cdot A\left(x-\frac{z}{2}\right)}u(x-z)- 2u(x)}{|z|^{\frac{3+2s}{2}}}\bessel{\frac{3+2s}{2}}{m|z|}1_{B^c_{\eps}(0)}(z) \\
			\to \frac{\e^{-iz\cdot A\left(x+\frac{z}{2}\right)}u(x+z) + \e^{iz\cdot A\left(x-\frac{z}{2}\right)}u(x-z)- 2u(x)}{|z|^{\frac{3+2s}{2}}}\bessel{\frac{3+2s}{2}}{m|z|} \text{ for a.e. $z \in \R^3$ as $\eps \to 0^+$, }
		\end{multline*}
		and, taking $R>0$ as in Lemma \ref{numerator:control:origin}\footnote{Here, since we are interested in the pointwise convergence, we can consider $K=\{x\}$ in Lemma \ref{numerator:control:origin}.}
		and using \eqref{Bessel:bound} 
		\begin{equation}
			\label{auxiliary:funct:integrability}
			\begin{aligned}
				&\frac{\left|\e^{-iz\cdot A\left(x+\frac{z}{2}\right)}u(x+z) + \e^{iz\cdot A\left(x-\frac{z}{2}\right)}u(x-z)- 2u(x)\right|}{|z|^{\frac{3+2s}{2}}}\bessel{\frac{3+2s}{2}}{m|z|}1_{B^c_{\eps}(0)}(z) \\
				& \qquad =
				\frac{\left|\e^{-iz\cdot A\left(x+\frac{z}{2}\right)}u(x+z) + \e^{iz\cdot A\left(x-\frac{z}{2}\right)}u(x-z)- 2u(x)\right|}{|z|^{\frac{3+2s}{2}}}\bessel{\frac{3+2s}{2}}{m|z|}1_{B^c_{\eps}(0)}(z)1_{B_R(0)}(z) \\
				& \qquad \qquad + \frac{\left|\e^{-iz\cdot A\left(x+\frac{z}{2}\right)}u(x+z) + \e^{iz\cdot A\left(x-\frac{z}{2}\right)}u(x-z)- 2u(x)\right|}{|z|^{\frac{3+2s}{2}}}\bessel{\frac{3+2s}{2}}{m|z|}1_{B^c_{\eps}(0)}(z)1_{B^c_R(0)}(z)\\
				& \qquad \leq \frac{C}{|z|^{\frac{3+2s}{2}-1-\gamma}}\bessel{\frac{3+2s}{2}}{m|z|}1_{B_R(0)}(z)  + \frac{4\|u\|_{L^{\infty}(\R^3,\C)}}{|z|^{\frac{3+2s}{2}}}\bessel{\frac{3+2s}{2}}{m|z|}1_{B^c_R(0)}(z)\\
				& \qquad \leq \frac{C}{|z|^{2+2s-\gamma}}1_{B_R(0)}(z) + \frac{4\|u\|_{L^{\infty}(\R^3,\C)}}{|z|^{3+2s}}1_{B^c_R(0)}(z) \in L^1(\R^3,\C),
			\end{aligned}
		\end{equation}
		since $\gamma > 2s-1$ and $s>0$, respectively. Therefore, for almost every $x \in \R^3$, by the Dominated Convergence Theorem,
		\begin{multline*}
			\left(-\Delta_A + m^2\right)^s u(x) =
			C_s{m}^{\frac{3+2s}{2}} \lim_{\eps \to 0^+}g_\eps (x) + m^{2s}u(x)\\
			=
			-\frac{C_s}{2}m^{\frac{3+2s}{2}}\int_{\R^3}\frac{\e^{-iz\cdot A\left(x +\frac{z}{2}\right)}u(x+z) + \e^{iz\cdot A\left(x -\frac{z}{2}\right)}u(x-z) - 2u(x)}{|z|^{\frac{3+2s}{2}}}\bessel{\frac{3+2s}{2}}{m|z|} \dd z
			+ m^{2s}u(x).
		\end{multline*}
	\end{proof}
	Now, we are ready to prove the main result of this subsection.
	\begin{proof}[Proof of Theorem \ref{ABMP:Th}]
		For any $u \in C^{\infty}_c(\R^3,\C)$, by Proposition \ref{remove:singularity}, we can write 
		\begin{align*}
			(-\Delta + m^2)_A^su(x) 
			=&-\frac{C_s}{2}m^{\frac{3+2s}{2}}\int_{B^c_1(0)}\frac{\e^{-iz\cdot A\left(x +\frac{z}{2}\right)}u(x+y) + \e^{iz\cdot A\left(x -\frac{z}{2}\right)}u(x-z) - 2u(x)}{|z|^{\frac{3+2s}{2}}}\bessel{\frac{3+2s}{2}}{m|z|} \dd z\\ 
			&-\frac{C_s}{2}m^{\frac{3+2s}{2}}\int_{B_1(0)}\frac{\e^{-iz\cdot A\left(x +\frac{z}{2}\right)}u(x+y) + \e^{iz\cdot A\left(x -\frac{z}{2}\right)}u(x-z) - 2u(x)}{|z|^{\frac{3+2s}{2}}}\bessel{\frac{3+2s}{2}}{m|z|} \dd z\\
			&+ m^{2s}u(x).
		\end{align*}
		Using the monotonicity of $\mathcal{K}_\nu$ with respect to $\nu$ and \eqref{Bessel:bound} we have
		\begin{multline*}
			\left|\frac{C_s}{2}{m}^{\frac{3+2s}{2}}\int_{B^c_1(0)} \frac{\e^{-iz \cdot A\left(x+\frac{z}{2}\right)}u(x+z) + \e^{iz \cdot A\left(x-\frac{z}{2}\right)}u(x-z) - 2u(x)}{|z|^{\frac{3+2s}{2}}}\bessel{\frac{3+2s}{2}}{m|z|} \dd z \right| \\
			\leq 2C_s{m}^{\frac{3+2s}{2}}\|u\|_{L^\infty(\R^3,\C)}\int_{B^c_1(0)} |z|^{-\frac{3+2s}{2}}\bessel{\frac{5}{2}}{m|z|} \dd z \leq C\frac{C_sm^{s-1}}{s+1}\|u\|_{L^\infty(\R^3,\C)}.
		\end{multline*}
		Since $\lim_{s \nearrow 1}\Gamma(1-s)=+\infty$ then, by \eqref{pse:rel:mag:con}, 
		\begin{equation}
			\label{fractional:constant:to:zero}
			\lim_{s \nearrow 1}C_s=0,
		\end{equation}
		and so it follows that
		\begin{equation*}
			\lim_{s \nearrow 1} \frac{C_s}{2}{m}^{\frac{3+2s}{2}} \int_{B^c_1(0)} \frac{\e^{-iz \cdot A\left(x+\frac{z}{2}\right)}u(x+z) + \e^{iz \cdot A\left(x-\frac{z}{2}\right)}u(x-z) - 2u(x)}{|z|^{\frac{3+2s}{2}}}\bessel{\frac{3+2s}{2}}{m|z|} \dd z = 0.
		\end{equation*}
		When we are close to zero, we need a careful analysis. First of all, setting
		\[
		f_1(z) = \e^{-iz\cdot A\left(x+\frac{z}{2}\right)}u(x+z),
		\]
		we observe that, for $k \in \set{1,2,3}$, the $k-$th partial derivative is
		\begin{equation*}
			\partial_k f_1(z) = \e^{-iz\cdot A\left(x+\frac{z}{2}\right)}\left[\left(-i A_k\left(x+\frac{z}{2}\right)-\frac{i}{2}z \cdot \partial_k A\left(x+\frac{z}{2}\right)\right)u(x+z) + \partial_k u(x+z)\right]
		\end{equation*}
		and, for $h \in \set{1,2,3}$, the second derivatives are given by
		\begin{equation*}
			\begin{aligned}
				\partial^2_{hk}f_1(z) & = \e^{-iz\cdot A\left(x+\frac{z}{2}\right)}\left[\partial^2_{hk}u(x+z) + \left(-iA_k\left(x+\frac{z}{2}\right) - \frac{i}{2}z \cdot \partial_kA\left(x+\frac{z}{2}\right)\right)\partial_hu(x+z) \right. \\
				& \left. + \left(-iA_h\left(x+\frac{z}{2}\right) - \frac{i}{2}z \cdot \partial_hA\left(x+\frac{z}{2}\right)\right)\partial_ku(x+z) \right. \\
				& \left. + \left(- \frac{i}{4}z \cdot \partial^2_{hk}A\left(x+\frac{z}{2}\right) -\frac{i}{2}\left(\partial_hA_k\left(x+\frac{z}{2}\right) +\partial_kA_h\left(x+\frac{z}{2}\right)\right) \right.\right. \\
				& \left.\left.- \left(A_k\left(x+\frac{z}{2}\right) + \frac{z}{2} \cdot \partial_kA\left(x+\frac{z}{2}\right)\right)\left(A_h\left(x+\frac{z}{2}\right) + \frac{z}{2} \cdot \partial_hA\left(x+\frac{z}{2}\right) \right) \right)u(x+z)\right]
			\end{aligned}
		\end{equation*}
		Thus, the Taylor's expansion, up to the second order, of $f_1(z)$ centered in $0$ gives
		\[
		f_1(z) = u(x) + \nabla_Au(x) \cdot z + \frac12 H_Au(x)z\cdot z + o(|z|^2),
		\]
		where $H_Au(x)$ is the $3 \times 3-$matrix whose general term is
		\[
		\mathfrak{h}_{hk}u(x) := \partial^2_{hk}u(x) - iA_k(x)\partial_hu(x) - iA_h(x)\partial_ku(x) - \frac{i}{2}\left(\partial_hA_k(x) + \partial_kA_h(x)\right)u(x) - A_k(x)A_h(x)u(x).
		\]
		Observe that $\operatorname{tr}H_Au = \Delta_Au$.\\
		Making the same computations for
		\[
		f_2(z) = \e^{iz\cdot A\left(x-\frac{z}{2}\right)}u(x-z)
		\]
		for $k \in \set{1,2,3}$, the $k-$th partial derivative is
		\begin{equation*}
			\partial_k f_2(z) = \left[\left(i A_k\left(x-\frac{z}{2}\right) - \frac{i}{2}z \cdot \partial_k A\left(x-\frac{z}{2}\right)\right)u(x-z) - \partial_k u(x-z)\right]\e^{iz\cdot A\left(x-\frac{z}{2}\right)}
		\end{equation*}
		and, for $h \in \set{1,2,3}$, the second derivatives are given by
		\begin{equation*}
			\begin{aligned}
				\partial^2_{hk}f_2(z) & = \e^{iz\cdot A\left(x-\frac{z}{2}\right)}\left[\partial^2_{hk}u(x-z) - \left(iA_k\left(x-\frac{z}{2}\right) - \frac{i}{2}z \cdot \partial_kA\left(x-\frac{z}{2}\right)\right)\partial_hu(x-z) \right. \\
				& \left. - \left(iA_h\left(x-\frac{z}{2}\right) - \frac{i}{2}z \cdot \partial_hA\left(x-\frac{z}{2}\right)\right)\partial_ku(x-z) \right. \\
				& \left. + \left(\frac{i}{4}z \cdot \partial^2_{hk}A\left(x-\frac{z}{2}\right) -\frac{i}{2}\left(\partial_hA_k\left(x-\frac{z}{2}\right) +\partial_kA_h\left(x-\frac{z}{2}\right)\right) \right.\right. \\
				& \left.\left.- \left(A_k\left(x-\frac{z}{2}\right) - \frac{z}{2} \cdot \partial_kA\left(x-\frac{z}{2}\right)\right)\left(A_h\left(x-\frac{z}{2}\right) - \frac{z}{2} \cdot \partial_hA\left(x-\frac{z}{2}\right) \right) \right)u(x-z)\right]
			\end{aligned}
		\end{equation*}
		we obtain,
		\[
		f_2(z) = u(x) - \nabla_Au(x) \cdot z + \frac12 H_Au(x)z\cdot z + o(|z|^2).
		\]
		Therefore, we have that
		\[
		\e^{-iz \cdot A\left(x+\frac{z}{2}\right)}u(x+z) + \e^{iz \cdot A\left(x-\frac{z}{2}\right)}u(x-z) - 2u(x) = H_Au(x)z\cdot z + o(|z|^2)
		\]
		and so, there exists $C>0$ such that
		\[
		\left|\e^{-iz \cdot A\left(x+\frac{z}{2}\right)}u(x+z) + \e^{iz \cdot A\left(x-\frac{z}{2}\right)}u(x-z) - 2u(x) - H_Au(x)z\cdot z\right| \leq C|z|^3 
		\]
		for $z \in B_1(0)$.\\
		Hence, using the monotonicity of $\mathcal{K}_\nu$ with respect to $\nu$ and \eqref{Bessel:bound}, we obtain
		\begin{multline*}
			\left|\int_{B_1(0)} \frac{\e^{-iz \cdot A\left(x+\frac{z}{2}\right)}u(x+z) + \e^{iz \cdot A\left(x-\frac{z}{2}\right)}u(x-z) - 2u(x) - H_Au(x)z \cdot z}{|z|^{\frac{3+2s}{2}}}\bessel{\frac{3+2s}{2}}{m|z|} \dd z \right| \\
			\leq C \int_{B_1(0)} |z|^{\frac{3-2s}{2}} \bessel{\frac{5}{2}}{m|z|} \dd z 
			\leq C \int_0^1 r^{1-s} \dd r = \frac{C}{2-s}.
		\end{multline*}
		Therefore, using again \eqref{fractional:constant:to:zero}, we obtain
		\begin{equation}
			\label{limit:on:the:ball}
			\begin{split}
				&\frac{C_s}{2}{m}^{\frac{3+2s}{2}}\int_{B_1(0)} \frac{\e^{-iz \cdot A\left(x+\frac{z}{2}\right)}u(x+z) + \e^{iz \cdot A\left(x-\frac{z}{2}\right)}u(x-z) - 2u(x) - H_Au(x)z\cdot z}{|z|^{\frac{3+2s}{2}}}\bessel{\frac{3+2s}{2}}{m|z|} \dd z \\&\qquad \to 0 \text{ as } s \nearrow 1. 
			\end{split}
		\end{equation}
		Thus, since, due to the oddness of the function $\mathfrak{h}_{hk}u(x)z_hz_k$ for $h \neq k$, we have
		\begin{equation*}
			\begin{aligned}
				&\lim_{s \nearrow 1} \frac{C_s}{2}{m}^{\frac{3+2s}{2}} \int_{B_1(0)} \frac{H_Au(x)z \cdot z}{|z|^{\frac{3+2s}{2}}}\bessel{\frac{3+2s}{2}}{m|z|} \dd z \\
				& \quad = \lim_{s \nearrow 1} \frac{C_s}{2}{m}^{\frac{3+2s}{2}} \sum_{h = 1}^3\int_{B_1(0)} \frac{\mathfrak{h}_{hh}u(x)z^2_h}{|z|^{\frac{3+2s}{2}}}\bessel{\frac{3+2s}{2}}{m|z|} \dd z \\
				& \quad = \tr H_Au(x)\lim_{s \nearrow 1} \frac{C_s}{2}{m}^{\frac{3+2s}{2}} \frac13\sum_{l = 1}^3\int_{B_1(0)} \frac{z^2_l}{|z|^{\frac{3+2s}{2}}}\bessel{\frac{3+2s}{2}}{m|z|} \dd z \\ 
				& \quad = \Delta_A u(x)\lim_{s \nearrow 1} \left(\frac{2\pi}{3}C_s{m}^{\frac{3+2s}{2}}\int_0^1 r^{\frac{5-2s}{2}}\bessel{\frac{3+2s}{2}}{mr} \dd r\right),
			\end{aligned}
		\end{equation*}
		then, from \eqref{limit:on:the:ball}, we obtain
		\begin{equation*}
			\begin{aligned}
				&-\lim_{s \nearrow 1} \frac{C_s}{2}{m}^{\frac{3+2s}{2}} \int_{B_1(0)} \frac{\e^{-iz \cdot A\left(x+\frac{z}{2}\right)}u(x+z) + \e^{iz \cdot A\left(x-\frac{z}{2}\right)}u(x-z) - 2u(x)}{|z|^{\frac{3+2s}{2}}}\bessel{\frac{3+2s}{2}}{m|z|} \dd z \\
				& \qquad = -\Delta_A u(x)\lim_{s \nearrow 1} \left(\frac{2\pi}{3}C_s{m}^{\frac{3+2s}{2}}\int_0^1 r^{\frac{5-2s}{2}}\bessel{\frac{3+2s}{2}}{mr} \dd r\right).
			\end{aligned}
		\end{equation*}
		But, 
		\[
		\lim_{s \nearrow 1} \left(\frac{2\pi}{3}C_s{m}^{\frac{3+2s}{2}}\int_0^1 r^{\frac{5-2s}{2}}\bessel{\frac{3+2s}{2}}{mr} \dd r\right) = 1.
		\]
		Indeed, by \eqref{I1:integral},
		\[
		\frac{2\pi}{3}C_sm^{\frac{3+2s}{2}}\int_0^{+\infty} r^{\frac{5-2s}{2}}\bessel{\frac{3+2s}{2}}{mr} \dd r = sm^{2(s-1)} \to 1 \text{ as } s \nearrow 1,
		\]
		and, arguing as in \eqref{I2:integral}, we get
		\[
		\frac{2\pi}{3}C_sm^{\frac{3+2s}{2}}\int_1^{+\infty}r^{\frac{5-2s}{2}}\bessel{\frac{3+2s}{2}}{mr} \dd r \to 0 \text{ as } s \nearrow 1,
		\]
		so that we can conclude.
	\end{proof}

	\section{Weak to strong}
	\label{section:pointwise:representation}
	In this section, we give the notion of weak solution and we show that, under suitable assumptions, such kind of solutions satisfies the equation pointwise.
	
	Let $H^{-s}_A(\R^3,\C)$ be the dual space of $H^s_A(\R^3,\C)$. 
	\begin{Def}
		If $f \in H^{-s}_A(\R^3,\C)$, then $u \in H^s_A(\R^3,\C)$ is a {\em weak solution} to
		\begin{equation}
			\label{SME}
			(-\Delta + m^2)_A^su = f \quad \text{ in } \R^3
		\end{equation}
		if, for every $v \in H^s_A(\R^3,\C)$,
		\begin{multline}
			\label{weak:sol}
			\frac{C_s}{2}m^{\frac{3+2s}{2}}\Re\int_{\R^3 \times \R^3} \frac{\left(\e^{-i(x-y)\cdot A\left(\frac{x+y}{2}\right)}u(x) - u(y)\right)\overline{\left(\e^{-i(x-y)\cdot A\left(\frac{x+y}{2}\right)}v(x) - v(y)\right)}}{|x-y|^{\frac{3+2s}{2}}}\bessel{\frac{3+2s}{2}}{m|x-y|} \dd x \dd y\\
			\qquad + m^{2s}\Re\int_{\R^3}u(x)\overline{v(x)} \dd x = \Re\int_{\R^3}f(x)\overline{v(x)} \dd x. 
		\end{multline}
	\end{Def}
	Then, we have
	\begin{Th}[Pointwise representation]
		\label{th:pointwise:representation}
		Let $u \in H^s_A(\R^3,\C)$ be a weak solution to \eqref{SME}. Assume that $A \in C^2(\R^3,\C)$ and $u \in L^{\infty}(\R^3,\C) \cap C^{1,\gamma}_{\loc}(\R^3,\C)$, for some $\gamma \in (0,1]$ with $\gamma > 2s-1$. Then, $u$ solves \eqref{SME} pointwise a.e. in $\R^3$.
	\end{Th}
	\begin{proof}
		Let us consider equation \eqref{weak:sol} for $v \in C^{\infty}_c(\R^3,\C)$ and let $K:=\supp v$. Since
		\[
		C_sm^{\frac{3+2s}{2}}g_{\eps}(x) + m^{2s}u(x) \to (-\Delta + m^2)_A^su(x) \text{ a.e. in } \R^3 \text{ as } \eps \to 0^+,
		\]
		where $g_{\eps}$ is the auxiliary function defined in \eqref{auxiliary:function}, see Proposition \ref{remove:singularity}, and, arguing as in \eqref{auxiliary:funct:integrability}, 
		\[
		|g_{\eps}(x)| \leq C\left(\int_{\R^3}\frac{1}{|z|^{2+2s-\gamma}}1_{B_R(0)}(z) \dd z + \int_{\R^3}\frac{\|u\|_{L^{\infty}(\R^3,\C)}}{|z|^{3+2s}}1_{B^c_R(0)}(z) \dd z \right) \leq C \text{ for } x \in K,
		\]
		and by the Dominated Convergence Theorem,
		\begin{equation}
			\label{L1:strong:convg}
			C_sm^{\frac{3+2s}{2}}g_{\eps} + m^{2s}u \to (-\Delta + m^2)_A^su \text{ in } L^1(K) \text{ as } \eps \to 0^+.
		\end{equation}
		Thus, since
		\begin{equation*}
			\begin{aligned}
				&-\left(u(y) - \e^{i(y-x)\cdot A\left(\frac{y+x}{2}\right)}u(x)\right)\overline{\e^{i(y-x)\cdot A\left(\frac{y+x}{2}\right)}v(x)}\bessel{\frac{3+2s}{2}}{m|y-x|}1_{B^c_{\eps}(y)}(x)\\
				& \qquad = \left(u(x) - \e^{i(x-y)\cdot A\left(\frac{x+y}{2}\right)}u(y)\right)\overline{v(x)}\bessel{\frac{3+2s}{2}}{m|x-y|}1_{B^c_{\eps}(x)}(y),
			\end{aligned}
		\end{equation*}
		then, by Fubini-Tonelli Theorem,
		\begin{align*}
			&\Re\int_{\R^3 \times \R^3} \frac{\left(\e^{-i(x-y)\cdot A\left(\frac{x+y}{2}\right)}u(x) - u(y)\right)\overline{\left(\e^{-i(x-y)\cdot A\left(\frac{x+y}{2}\right)}v(x) - v(y)\right)}}{|x-y|^{\frac{3+2s}{2}}}\bessel{\frac{3+2s}{2}}{m|x-y|} \dd x \dd y\\
			& = \Re\int_{\R^3 \times \R^3} \frac{\left(u(x) - \e^{i(x-y)\cdot A\left(\frac{x+y}{2}\right)}u(y)\right)\overline{\left(v(x) - \e^{i(x-y)\cdot A\left(\frac{x+y}{2}\right)}v(y)\right)}}{|x-y|^{\frac{3+2s}{2}}}\bessel{\frac{3+2s}{2}}{m|x-y|} \dd x \dd y\\
			& = \lim_{\eps \to 0^+}\Re\int_{\R^3 \times \R^3} \frac{\left(u(x) - \e^{i(x-y)\cdot A\left(\frac{x+y}{2}\right)}u(y)\right)\overline{\left(v(x) - \e^{i(x-y)\cdot A\left(\frac{x+y}{2}\right)}v(y)\right)}}{|x-y|^{\frac{3+2s}{2}}}\bessel{\frac{3+2s}{2}}{m|x-y|}1_{B^c_{\eps}(x)}(y) \dd x \dd y\\
			& = \lim_{\eps \to 0^+}\left( \Re\int_{\R^3 \times \R^3} \frac{\left(u(x) - \e^{i(x-y)\cdot A\left(\frac{x+y}{2}\right)}u(y)\right)\overline{v(x)}}{|x-y|^{\frac{3+2s}{2}}}\bessel{\frac{3+2s}{2}}{m|x-y|}1_{B^c_{\eps}(x)}(y) \dd x \dd y \right.\\
			& \qquad - \left.\Re\int_{\R^3 \times \R^3} \frac{\left(u(x) - \e^{i(x-y)\cdot A\left(\frac{x+y}{2}\right)}u(y)\right)\overline{\e^{i(x-y)\cdot A\left(\frac{x+y}{2}\right)}v(y)}}{|x-y|^{\frac{3+2s}{2}}}\bessel{\frac{3+2s}{2}}{m|x-y|}1_{B^c_{\eps}(x)}(y) \dd x \dd y\right)\\
			& = \lim_{\eps \to 0^+}\left(\Re\int_{\R^3}g_{\eps}(x)\overline{v(x)} \dd x \right.\\
			& \qquad - \left.\Re\int_{\R^3 \times \R^3} \frac{\left(u(y) - \e^{i(y-x)\cdot A\left(\frac{y+x}{2}\right)}u(x)\right)\overline{\e^{i(y-x)\cdot A\left(\frac{y+x}{2}\right)}v(x)}}{|y-x|^{\frac{3+2s}{2}}}\bessel{\frac{3+2s}{2}}{m|x-y|}1_{B^c_{\eps}(y)}(x) \dd y \dd x\right)\\
			& = \lim_{\eps \to 0^+} \left(\Re\int_{\R^3}g_{\eps}(x)\overline{v(x)} \dd x \right.\\
			& \qquad \left. + \Re\int_{\R^3 \times \R^3} \frac{\left(u(x) - \e^{i(x-y)\cdot A\left(\frac{x+y}{2}\right)}u(y)\right)\overline{v(x)}}{|x-y|^{\frac{3+2s}{2}}}\bessel{\frac{3+2s}{2}}{m|x-y|}1_{B^c_{\eps}(x)}(y) \dd x \dd y \right) \\
			& = 2\lim_{\eps \to 0^+} \Re\int_{\R^3}g_{\eps}(x)\overline{v(x)} \dd x.
		\end{align*}
		Therefore, using \eqref{L1:strong:convg}, we get
		\begin{align*}
			&\frac{C_s}{2}m^{\frac{3+2s}{2}}\Re\int_{\R^3 \times \R^3} \frac{\left(\e^{-i(x-y)\cdot A\left(\frac{x+y}{2}\right)}u(x) - u(y)\right)\overline{\left(\e^{-i(x-y)\cdot A\left(\frac{x+y}{2}\right)}v(x) - v(y)\right)}}{|x-y|^{\frac{3+2s}{2}}}\bessel{\frac{3+2s}{2}}{m|x-y|} \dd x \dd y\\
			& \qquad\qquad + m^{2s} \Re\int_{\R^3}u(x)\overline{v(x)} \dd x \\
			& \qquad= C_sm^{\frac{3+2s}{2}}\lim_{\eps \to 0^+} \Re\int_{\R^3}g_{\eps}(x)\overline{v(x)} \dd x + m^{2s} \Re\int_{\R^3}u(x)\overline{v(x)} \dd x\\
			& \qquad = \Re\int_{\R^3}(-\Delta + m^2)_A^su(x)\overline{v(x)} \dd x
		\end{align*}
		and so, by \eqref{weak:sol}, for all $v \in C^{\infty}_c(\R^3,\C)$,
		\begin{align*}
			0&=\frac{C_s}{2}m^{\frac{3+2s}{2}}\Re\int_{\R^3 \times \R^3} \frac{\left(\e^{-i(x-y)\cdot A\left(\frac{x+y}{2}\right)}u(x) - u(y)\right)\overline{\left(\e^{-i(x-y)\cdot A\left(\frac{x+y}{2}\right)}v(x) - v(y)\right)}}{|x-y|^{\frac{3+2s}{2}}}\bessel{\frac{3+2s}{2}}{m|x-y|} \dd x \dd y\\
			& \qquad + m^{2s}\Re\int_{\R^3}u(x)\overline{v(x)} \dd x - \Re\int_{\R^3}f(x)\overline{v(x)} \dd x \\
			& = \Re\int_{\R^3}(-\Delta + m^2)_A^su(x)\overline{v(x)} \dd x - \Re\int_{\R^3}f(x)\overline{v(x)} \dd x.    
		\end{align*}
		Hence
		\[
		(-\Delta + m^2)_A^su = f \text{ a.e. in } \R^3.
		\]
	\end{proof}

	\section{Applications to some semilinear equations}
	\label{Applications}
	
	Taking into account the functional setting introduced in Section \ref{Functional:setting:sec}, our aim is now to find ground state solutions to \eqref{eq:semilinear} and \eqref{eq:Choquard}.

	The existence results mentioned in the introduction will be found in the radially symmetric setting and in the general one.
	
	Moreover, the magnetic potential will play a crucial role on these arguments.\\
	In the radial case, to apply the Principle of Symmetric Criticality, we will suppose the potential $A$ satisfies the condition
	\begin{equation}
		\label{isometry}
		g(x-y) \cdot A\left(g\frac{x+y}{2}\right) = (x-y) \cdot A\left(\frac{x+y}{2}\right) \text{ for all } x,y \in \R^3, g \in O(3).
	\end{equation}
	Observe that, condition \eqref{isometry} is satisfied if $A$ is equivariant and so, for instance, if $A(x) = a(|x|)x$, where $a:[0,+\infty) \to \R$ and $x \in \R^3$ (see also \cite[Lemma 2.1]{BaDoPlRe}).
	
	If no symmetry assumption is assumed, we will consider the following two cases.\\
	In the first one, we consider linear potential, i.e.
	\begin{equation}
		\label{hyp:A:linear}\tag{$\mathcal{A}_\ell$}
		A(x)=Ax, \, x \in \R^3, \text{ where $A$ is a $3 \times 3$ real symmetric constant matrix}\footnote{In such a case the magnetic field $B=\nabla\times A=0$.}.
	\end{equation}
	In the second one, we ask for a good behavior of the potential at infinity, in the sense of the following definition (see \cite[Defintion 4.1]{EstebanLions} and \cite[Definition 1.1]{dASq}).
	\begin{Def}
		\label{assumption:A}
		We say that $A$ satisfies assumption {\rm (\hypertarget{Ab}{$\cA_b$})}, if for any unbounded sequence $\Xi=\{\xi_n\}_n \subset \R^3$ there exists a sequence $\{H_n\}_n \subset \R^3$ and a function $A_\Xi : \R^3 \to \R^3$ such that:
		\begin{enumerate}
			\item $\lim_n A_n(x) = A_\Xi(x)$ for all $x \in \R^3$, 
			\item $\sup_{n \in \N}\|A_n\|_{L^{\infty}(K,\R^3)} < +\infty$ for every compact set $K$, 
		\end{enumerate}
		where $A_n(x):=A(x + \xi_n) + H_n$ and $\{\xi_n\}_n$ is a subsequence of $\Xi$ such that $|\xi_n| \to +\infty$.
	\end{Def}
	Observe that, if $A$ admits limit as $|x| \to +\infty$, then it satisfies assumption {\rm (\hyperlink{Ab}{$\cA_b$})
	}. 
	
	The rest of this section will be divided in two parts, in which we separately consider the two equations \eqref{eq:semilinear} and \eqref{eq:Choquard} respectively, introducing firstly the minimization problem that we want to solve, and then stating and proving the existence results.

	\subsection{The pure power nonlinearity}
	\label{subsec:semilinear}
	
	Let us formally introduce the minimization problem associated to the equation \eqref{eq:semilinear}. Its weak solutions will be sought as minimizers of the $H^s_A(\R^3,\C)$ norm on $\cS:=\left\{u \in H^s_A(\R^3,\C) : \|u\|_{L^p(\R^3,\C)}=1\right\}$.
	Indeed, if such minimizers exist, then by Multiplier Lagrange Theorem and using a rescaling argument, we obtain (weak) solutions to \eqref{eq:semilinear}. Moreover, in this type of problems, it is well known that such minimizers give rise to {\em ground state solutions}.
	
	Let 
	\begin{equation*}
		\cM_A:=\inf_{u \in \cS} \|u\|^2_{A,s}. 
	\end{equation*}
	
	Observe that $\cM_A > 0$.
	Indeed, by Lemma \ref{magn:sob:emb}, there exists $C>0$ such that for every $u \in \cS$
	\[
	\|u\|^2_{A,s} \geq C^2\|u\|^2_{L^p(\R^3,\C)} = C^2>0.
	\]
	
	Moreover, setting 
	\begin{equation*}
		\cM_{\lambda,A} := \inf_{u \in \cS_{\lambda}}\|u\|_{A,s}^2,
		\quad
		\text{where }
		\cS_{\lambda} :=\left\{u \in H^s_A(\R^3,\C) : \|u\|^p_{L^p(\R^3,\C)} = \lambda \right\} \text{ and } \lambda > 0,
	\end{equation*}
	it holds 
	\begin{equation}
		\label{rescaled:min:prob}
		\cM_{\lambda,A} = \lambda^{\frac{2}{p}}\cM_A.
	\end{equation}
	
	Now that the minimization problem has been formally settled, we can focus on the existence results for \eqref{eq:semilinear}, starting with the radial case. Of course, in this setting we assume \eqref{isometry}. First of all, we consider the restriction of our minimization problem to the radial functions: that is, we set
	\[
	\cS_{\rad} := \cS \cap H^s_{A,\rad}(\R^3,\C)
	\]
	and we want to minimize the norm of our space on $\cS_{\rad}$, i.e. we want to show that
	\begin{equation*}
		\cM_{A,\rad} := \inf_{u \in \cS_{\rad}}\|u\|^2_{A,s}
	\end{equation*}
	is attained.
	
	Exploiting the radiality property, we are able to recover the strong convergence of the minimizing sequences, making the first existence result quite straightforward.
	\begin{Th}
		\label{semilinear:symmetric:existence:th}
		Let $A:\R^3 \to \R^3$ be a vector field with locally bounded gradient and satisfying \eqref{isometry}. For any $p \in (2,2^*_s)$ there exists a nontrivial weak solution $u \in H^s_{A,\rad}(\R^3,\C)$ to \eqref{eq:semilinear}.
	\end{Th}
	\begin{proof}
		Let $\{u_n\}_n \subset \cS_{\rad}$ be a minimizing sequence, namely $\|u_n\|^2_{A,s} \to \cM_{A,\rad}$ as $n \to +\infty$. Therefore, up to a subsequence, $u_n \weakto u$ in $H^s_{A,\rad}(\R^3,\C)$, and so, by Lemma \ref{magn:sob:emb}, $u_n \to u$ almost everywhere, by Lemma \ref{radial:comp:emb}, $|u_n| \to v$ in $L^p(\R^3,\R)$, and, finally, due also to the previous pointwise convergence, $v=|u|$. Hence, $\|u\|_{L^p(\R^3,\C)}=1$ and 
		\[
		\cM_{A,\rad} \leq \|u\|^2_{A,s} \leq \liminf_n\|u_n\|^2_{A,s} = \cM_{A,\rad},
		\]
		concluding the proof.
	\end{proof}
	
	Now, we move on the setting where no symmetry assumption is assumed. The situation here is a bit harder, and it requires a more careful analysis of the problem. Indeed, in order to recover the strong convergence of the minimizing sequences, we rely on a concentration-compactness argument showing that both vanishing and dichotomy can not occur. To this end, first we need to prove the following technical Lemma.
	\begin{Lem}
		\label{dichotomy:lemma}
		Let $\{u_n\}_n \subset \cS$ such that $\lim_n\|u_n\|^2_{A,s}=L$ for some $L>0$ and
		\[
		\mu_n(x):=m^{2s}|u_n(x)|^2 + \frac{C_s}{2}m^{\frac{3+2s}{2}}\int_{\R^3} \frac{|\e^{-i(x-y)\cdot A\left(\frac{x+y}{2}\right)}u_n(x) - u_n(y) |^2}{|x-y|^{\frac{3+2s}{2}}}\mathcal{K}_{{\frac{3+2s}{2}}}(m|x-y|) \dd y,
		\quad
		x \in \R^3, n \in \N.
		\]	
		Assume that there is $\beta \in (0,L)$ such that for every $\eps>0$ there exist $\bar{R}>0$, $\overline{n} \geq 1$, a sequence of radii $\{R_n\}_n$ such that $R_n \to +\infty$, and a sequence $\{\xi_n\}_n \subset \R^3$ such that, defining $\mu^1_n:=1_{B_{\bar{R}}(\xi_n)}\mu_n$ and $\mu^2_n:=1_{B_{R_n}^c(\xi_n)}\mu_n$, for every $n \geq \overline{n}$,
		\begin{gather}
			\left|\int_{\R^3}\mu^1_n(x) \dd x - \beta \right| \leq \eps,\label{dichotomy:hyp:1}\\
			\left|\int_{\R^3}\mu^2_n(x) \,dx - (L-\beta) \right| \leq \eps,\label{dichotomy:hyp:2}\\
			\int_{\R^3}\left|\mu_n(x) - \mu^1_n(x)  - \mu^2_n(x) \right| \dd x\leq \eps. \label{dichotomy:hyp:3}
		\end{gather}
		Then, there exist two sequences $\{u^1_n\}_n,\{u^2_n\}_n \subset H^s_A(\R^3,\C)$ such that
		\begin{equation}
			\label{dichotomy:thesis:0}
			\lim_n \dist(\supp(u^1_n),\supp(u^2_n)) = +\infty
		\end{equation}
		and, for any $n \geq \overline{n}$,
		\begin{gather}
			\left|\|u^1_n\|^2_{A,s} - \beta\right| \leq \eps, \label{dichotomy:thesis:1}\\
			\left|\|u^2_n\|^2_{A,s} - (L-\beta)\right| \leq \eps, \label{dichotomy:thesis:2}\\
			\|u_n - u^1_n - u^2_n\|^2_{A,s} \leq \eps, \label{dichotomy:thesis:3}\\
			\left|1 - \|u^1_n\|^p_{L^p(\R^3,\C)} - \|u^2_n\|^p_{L^p(\R^3,\C)}\right| \leq \eps.\label{dichotomy:thesis:4}
		\end{gather}
	\end{Lem}
	\begin{proof}
		Let $\eps>0$. Consider $\varphi \in C^{\infty}(\R^3,\R)$ a radially symmetric function, $0 \leq \vp \leq 1$,
		\[
		\varphi(x)=
		\begin{cases}
			1, &\text{ if } x \in B_1(0),\\
			0, &\text{ if } x \in B^c_2(0),
		\end{cases}
		\]
		$\vp_r:=\vp(\cdot/r)$, for $r>0$, and 
		\[
		u^1_n:=\varphi_{\bar{R}}(\cdot-\xi_n)u_n \in H^s_A(\R^3,\C), \qquad u^2_n:=(1-\varphi_{R_n/2}(\cdot-\xi_n))u_n \in H^s_A(\R^3,\C).
		\]
		Since $R_n \to +\infty$, the Lipschitz constants of $\vp_{\bar{R}}(\cdot - \xi_n)$ and $1-\varphi_{R_n/2}(\cdot-\xi_n)$ are uniformly bounded with respect to $n$ and $\dist(\supp(u^1_n),\supp(u^2_n)) \to +\infty$ as $n \to +\infty$.\\ 
		Moreover, observe that, for $n$ large enough,
		\begin{equation}
			\label{u1mu1}
			\|u^1_n\|^2_{A,s} = \cI_n + \cJ_n
			\quad
			\text{and}
			\quad
			\int_{\R^3} \mu_n^1(x) dx = \cI_n + \tilde{\cJ}_n,
		\end{equation}
		where
		\begin{equation*}
			\begin{split}
				\cI_n
				&:= \frac{C_s}{2}m^{\frac{3+2s}{2}}\int_{B_{\bar{R}}(\xi_n) \times B_{\bar{R}}(\xi_n)}\frac{|\e^{-i(x-y)\cdot A\left(\frac{x+y}{2}\right)}u_n(x) - u_n(y) |^2}{|x-y|^{\frac{3+2s}{2}}}\mathcal{K}_{{\frac{3+2s}{2}}}(m|x-y|) \dd x \dd y \\
				&\quad + m^{2s}\int_{B_{\bar{R}}(\xi_n)}|u_n(x)|^2 \dd x,
			\end{split}
		\end{equation*}
		\begin{equation*}
			\begin{split}
				\cJ_n
				&:= \frac{C_s}{2}m^{\frac{3+2s}{2}}\int_{\left(B_{2\bar{R}}(\xi_n) \setminus B_{\bar{R}}(\xi_n)\right)^2}\frac{|\e^{-i(x-y)\cdot A\left(\frac{x+y}{2}\right)}u^1_n(x) - u^1_n(y) |^2}{|x-y|^{\frac{3+2s}{2}}}\mathcal{K}_{{\frac{3+2s}{2}}}(m|x-y|) \dd x \dd y \\
				&\quad +  C_sm^{\frac{3+2s}{2}}\int_{B_{2\bar{R}}(\xi_n) \setminus B_{\bar{R}}(\xi_n) \times B_{\bar{R}}(\xi_n)}\frac{|\e^{-i(x-y)\cdot A\left(\frac{x+y}{2}\right)}u^1_n(x) - u^1_n(y) |^2}{|x-y|^{\frac{3+2s}{2}}}\mathcal{K}_{{\frac{3+2s}{2}}}(m|x-y|) \dd x \dd y \\
				&\quad + C_sm^{\frac{3+2s}{2}}\int_{B_{2\bar{R}}(\xi_n) \setminus B_{\bar{R}}(\xi_n) \times B^c_{2\bar{R}}(\xi_n)}\frac{|\e^{-i(x-y)\cdot A\left(\frac{x+y}{2}\right)}u^1_n(x) - u^1_n(y) |^2}{|x-y|^{\frac{3+2s}{2}}}\mathcal{K}_{{\frac{3+2s}{2}}}(m|x-y|) \dd x \dd y \\
				& \quad +C_sm^{\frac{3+2s}{2}}\int_{B_{\bar{R}}(\xi_n) \times B_{R_n}(\xi_n) \setminus B_{2\bar{R}}(\xi_n)}\frac{|\e^{-i(x-y)\cdot A\left(\frac{x+y}{2}\right)}u^1_n(x) - u^1_n(y) |^2}{|x-y|^{\frac{3+2s}{2}}}\mathcal{K}_{{\frac{3+2s}{2}}}(m|x-y|) \dd x \dd y \\
				&\quad + C_sm^{\frac{3+2s}{2}}\int_{B_{\bar{R}}(\xi_n) \times B^c_{R_n}(\xi_n)}\frac{|u_n(x)|^2}{|x-y|^{\frac{3+2s}{2}}}\mathcal{K}_{{\frac{3+2s}{2}}}(m|x-y|) \dd x \dd y
				+ m^{2s}\int_{B_{2\bar{R}}(\xi_n)\setminus B_{\bar{R}}(\xi_n)}|u_n^1(x)|^2 \dd x,
			\end{split}
		\end{equation*}
		and
		\begin{equation*}
			\begin{split}
				\tilde{\cJ}_n
				&:=
				\frac{C_s}{2}m^{\frac{3+2s}{2}}\int_{B_{\bar{R}}(\xi_n) \times B_{R_n}(\xi_n) \setminus B_{\bar{R}}(\xi_n)}\frac{|\e^{-i(x-y)\cdot A\left(\frac{x+y}{2}\right)}u_n(x) - u_n(y) |^2}{|x-y|^{\frac{3+2s}{2}}}\mathcal{K}_{{\frac{3+2s}{2}}}(m|x-y|) \dd x \dd y\\
				&\quad
				+\frac{C_s}{2}m^{\frac{3+2s}{2}}\int_{B_{\bar{R}}(\xi_n) \times B^c_{R_n}(\xi_n)}\frac{|\e^{-i(x-y)\cdot A\left(\frac{x+y}{2}\right)}u_n(x) - u_n(y) |^2}{|x-y|^{\frac{3+2s}{2}}}\mathcal{K}_{{\frac{3+2s}{2}}}(m|x-y|) \dd x \dd y.       
			\end{split}
		\end{equation*}
		We claim that 
		\begin{equation}
			\label{claim:Jn}
			\cJ_n, \tilde{\cJ}_n \leq C\eps,
		\end{equation} 
		where $C$ depends only on $s$, $m$ and the Lipschitz constant of $\vp_{\bar{R}}$.\\
		For the first four integrals of $\cJ_n$ and the first one of $\tilde{\cJ_n}$, we can proceed in the same way. Here we give the details only for the first one.
		Observing that $B_{2\bar{R}}(\xi_n) \setminus B_{\bar{R}}(\xi_n)\subset B_{R_n}(\xi_n) \setminus B_{\bar{R}}(\xi_n)$\footnote{$B_{R_n}(\xi_n) \setminus B_{2\bar{R}}(\xi_n)\subset B_{R_n}(\xi_n) \setminus B_{\bar{R}}(\xi_n)$ for the fourth integral of $\cJ_n$.} and applying Lemma \ref{cut-off} and \eqref{dichotomy:hyp:3},
		\begin{equation}
			\label{computations:details:1}
			\begin{split}
				&\int_{\left(B_{2\bar{R}}(\xi_n) \setminus B_{\bar{R}}(\xi_n)\right)^2}\frac{|\e^{-i(x-y)\cdot A\left(\frac{x+y}{2}\right)}u^1_n(x) - u^1_n(y) |^2}{|x-y|^{\frac{3+2s}{2}}}\mathcal{K}_{{\frac{3+2s}{2}}}(m|x-y|) \dd x \dd y\\
				&\qquad \leq \int_{B_{R_n}(\xi_n) \setminus B_{\bar{R}}(\xi_n) \times B_{2\bar{R}}(\xi_n) \setminus B_{\bar{R}}(\xi_n)}\frac{|\e^{-i(x-y)\cdot A\left(\frac{x+y}{2}\right)}u^1_n(x) - u^1_n(y) |^2}{|x-y|^{\frac{3+2s}{2}}}\mathcal{K}_{{\frac{3+2s}{2}}}(m|x-y|) \dd x \dd y \\
				& \qquad \leq C\left(\int_{B_{R_n}(\xi_n) \setminus B_{\bar{R}}(\xi_n)} |u_n(x)|^2 \dd x\right. \\
				& \qquad \qquad \left.+ \int_{B_{R_n}(\xi_n) \setminus B_{\bar{R}}(\xi_n) \times \R^3}\frac{|\e^{-i(x-y)\cdot A\left(\frac{x+y}{2}\right)}u_n(x) - u_n(y) |^2}{|x-y|^{\frac{3+2s}{2}}}\mathcal{K}_{{\frac{3+2s}{2}}}(m|x-y|) \dd x \dd y\right) \\
				& \qquad \leq C\int_{\R^3} \left( \mu_n(x) - \mu^1_n(x) - \mu^2_n(x)\right) \dd x \leq C\eps.
			\end{split}
		\end{equation}
		For the last integral of $\cJ_n$, since $\vp_{\bar{R}} \leq 1$, as above, by \eqref{dichotomy:hyp:3},
		\begin{equation}
			\label{computations:details:2}
			\int_{B_{2\bar{R}}(\xi_n)\setminus B_{\bar{R}}(\xi_n)}|u_n^1(x)|^2 \dd x \leq C\int_{\R^3} \left( \mu_n(x) - \mu^1_n(x) - \mu^2_n(x)\right) \dd x \leq C\eps.
		\end{equation}
		For the remaining integrals (the fifth integral of $\cJ_n$ and the second one of $\tilde{\cJ}_n$) we observe that, by \eqref{Bessel:bound}, since $|x-y| \geq R_n - \bar{R} \to +\infty$ as $n$ diverges, and the boundedness of $\{u_n\}_n$ in $L^2(\R^3,\C)$, 
		\begin{align*}
			&\int_{B_{\bar{R}}(\xi_n) \times B^c_{R_n}(\xi_n)}\frac{|u_n(x) |^2}{|x-y|^{\frac{3+2s}{2}}}\mathcal{K}_{{\frac{3+2s}{2}}}(m|x-y|) \dd x \dd y  \leq C\int_{B_{\bar{R}}(\xi_n)}|u_n(x)|^2 \left(\int_{B^c_{R_n}(\xi_n)}\frac{1}{|x-y|^{3+2s}} \dd y\right) \dd x \\
			& \qquad \leq \frac{1}{(R_n - \bar{R})^{\delta}}\int_{B_{\bar{R}}(\xi_n)}|u_n(x)|^2 \left(\int_{\{|x-y| \geq 1\}}\frac{1}{|x-y|^{3+2s-\delta}} \dd y\right) \dd x \leq \frac{C}{(R_n - \bar{R})^{\delta}} \leq C\eps,
		\end{align*}
		and
		\begin{align*}
			&\int_{B_{\bar{R}}(\xi_n) \times B^c_{R_n}(\xi_n)}\frac{|u_n(y) |^2}{|x-y|^{\frac{3+2s}{2}}}\mathcal{K}_{{\frac{3+2s}{2}}}(m|x-y|) \dd x \dd y  
			\leq
			C\int_{B^c_{R_n}(\xi_n)}|u_n(y)|^2 \left(\int_{B_{\bar{R}}(\xi_n)}\frac{1}{|x-y|^{3+2s}} \dd x\right) \dd y \\
			& \qquad \leq
			\frac{1}{(R_n - \bar{R})^{\delta}}
			\int_{B^c_{R_n}(\xi_n)}|u_n(y)|^2 \left(\int_{\{|x-y| \geq 1\}}\frac{1}{|x-y|^{3+2s-\delta}} \dd x\right) \dd y 
			\leq \frac{C}{(R_n - \bar{R})^{\delta}} \leq C\eps,
		\end{align*}
		where $\delta \in (0,2s)$. Thus, combining \eqref{dichotomy:hyp:1}, \eqref{u1mu1}, and \eqref{claim:Jn}, we get \eqref{dichotomy:thesis:1}.\\
		Arguing analogously for $u_n^2$ and $\mu_n^2$ and using the uniform boundedness of the Lipschitz constant of $\varphi_{R_n/2}$, from \eqref{dichotomy:hyp:2} we get \eqref{dichotomy:thesis:2}.
		\\
		Now, since $u_n - u^1_n - u^2_n=\tilde{\varphi} u_n$ where $\tilde{\varphi}: = \varphi_{R_n/2}(\cdot-\xi_n) - \varphi_{\bar{R}}(\cdot-\xi_n)\in C^{0,1}(\R^3,\R)$ and, for $n$ large enough, $0\leq \tilde{\varphi}\leq 1$, and $\tilde{\varphi}=0$ in $B_{\bar{R}}(\xi_n)\cup B_{R_n}^c(\xi_n)$, arguing as in \eqref{computations:details:1} and \eqref{computations:details:2} we obtain 
		\eqref{dichotomy:thesis:3}.\\
		To conclude, by the interpolation inequality, Lemma \ref{magn:sob:emb}, the boundedness of the norm, and arguing as in \eqref{computations:details:2}, 
		\begin{multline*}
			0
			\leq 1 - \|u^1_n\|^p_{L^p(\R^3,\C)} - \|u^2_n\|^p_{L^p(\R^3,\C)}
			\leq \int_{B_{R_n}(\xi_n) \setminus B_{\bar{R}}(\xi_n)} |u_n(x)|^p \dd x\\
			\leq \left(\int_{B_{R_n}(\xi_n) \setminus B_{\bar{R}}(\xi_n)} |u_n(x)|^2 \dd x\right)^{\frac{\theta}{2}}\|u_n\|^{1-\theta}_{L^{2^*_s}(\R^3,\C)}
			\leq C\eps
		\end{multline*}
		where $\theta \in (0,1)$, and so \eqref{dichotomy:thesis:4} follows.
	\end{proof}
	
	We are now ready to show the existence of a solution to \eqref{eq:semilinear} in the general setting (without any symmetry assumptions).
	\begin{Th}
		\label{semilinear:existence:th}
		Suppose that $A:\R^3 \to \R^3$ satisfies \eqref{hyp:A:linear} or is a vector field with locally bounded gradient such that {\rm (\hyperlink{Ab}{$\cA_b$})} holds and $\cM_A < \inf \cM_{A_{\Xi}}$.
		Then, for any $p \in (2,2^*_s)$, there exists a solution $u \in H^s_A(\R^3,\C)$ to \eqref{eq:semilinear}.
	\end{Th}
	\begin{proof}
		Consider a sequence $\{u_n\}_n \subset \cS$ such that  $\|u_n\|^2_{A,s} \to \cM_A$ as $n \to +\infty$ and $\{\mu_n\}_n$ as in Lemma \ref{dichotomy:lemma}. Our goal is to get compactness of $\{u_n\}_n$ by applying the Concentration-Compactness Lemma \cite[Lemma I.1]{Lions1984} and ruling out both dichotomy and vanishing.\\
		\textit{Vanishing does not hold.} Assume by contradiction that for all $R>0$
		\[
		\lim_n\sup_{\xi \in \R^3}\int_{B_R(\xi)} \mu_n(x) \dd x = 0.
		\]
		Then, by \eqref{pointwise:diamagnetic} and Lemma \ref{local:Hs:equivalence},
		\[
		\lim_n\sup_{\xi \in \R^3}\||u_n|\|^2_{H^s(B_R(\xi),\R)} = 0
		\]
		and so, by Lemma \ref{Lemma36}, for any $R>0$
		\[
		\lim_n\sup_{\xi \in \R^3}\int_{B_R(\xi)}|u_n(x)|^p \dd x = 0.
		\]
		Hence, by Lemma \ref{vanishing:lemma}, $u_n \to 0$ in $L^p(\R^3,\R)$, which contradicts $\{u_n\}_n \subset \cS$.\\	
		\textit{Dichotomy does not hold.} Assume by contradiction that there exists $\beta \in (0,\cM_A)$ such that, for every $\eps>0$, there are $\bar{R}>0$, $\bar{n} \geq 1$, a sequence of radii $R_n \to +\infty$, and a sequence $\{\xi_n\}_n \subset \R^3$ such that, if $\mu^1_n(x):=1_{B_{\bar{R}}(\xi_n)}\mu_n$ and $\mu^2_n(x):=1_{B^c_{R_n}(\xi_n)}\mu_n$, for every $n \geq \bar{n}$ \eqref{dichotomy:hyp:1}--\eqref{dichotomy:hyp:3} hold for $L=\cM_A$.\\
		Hence, by Lemma \ref{dichotomy:lemma} there exist two sequences $\{u^1_n\}_n, \{u^2_n\}_n \subset H^s_A(\R^3,\C)$ such that, for $n \geq \bar{n}$, \eqref{dichotomy:thesis:0}--\eqref{dichotomy:thesis:4} hold for $L=\cM_A$.\\
		Then there exist $\theta_{\eps}, \omega_{\eps} \in (0,1)$ such that, up to a subsequence, as $n \to +\infty$,
		\begin{gather}
			\theta_{n,\eps}:=\|u^1_n\|^p_{L^p(\R^3,\C)} \to \theta_{\eps},\label{theta:eps}\\
			\omega_{n,\eps}:=\|u^2_n\|^p_{L^p(\R^3,\C)} \to \omega_{\eps},\label{omega:eps}\\
			\left|1-\theta_{\eps} - \omega_{\eps}\right| \leq \eps.\label{1-theta-omega}
		\end{gather}
		Suppose by contradiction that $\lim_{\eps \to 0^+}\theta_{\eps} = 1$, then by \eqref{dichotomy:thesis:1}, \eqref{rescaled:min:prob}, and \eqref{theta:eps}, for $\eps>0$ small enough 
		\[
		\beta + \eps
		\geq \limsup_n\|u^1_n\|^2_{A,s}
		\geq \limsup_n\cM_{\theta_{n,\eps},A} 
		= \theta_{\eps}^{2/p}\cM_A
		> \beta + \eps.
		\]
		Now, suppose that $\lim_{\eps \to 0^+}\theta_{\eps} = 0$. Then $\omega_{\eps} \to 1$ and we can reach a contradiction reasoning as above with $u^2_n$ using \eqref{dichotomy:thesis:2}, \eqref{rescaled:min:prob}, and \eqref{omega:eps}.\\
		Therefore, by \eqref{dichotomy:thesis:1}, \eqref{dichotomy:thesis:2}, \eqref{theta:eps}--\eqref{1-theta-omega}, for $\eps>0$ small enough and since $\lambda^{2/p} + (1-\lambda)^{2/p} > 1$ for every $\lambda \in (0,1)$,
		\[
		\cM_A + 2\eps \geq \limsup_n\left(\|u^1_n\|^2_{A,s} + \|u^2_n\|^2_{A,s}\right) \geq \limsup_n\left(\cM_{\theta_{n,\eps},A} + \cM_{\omega_{n,\eps},A}\right) = \cM_A\left(\theta_{\eps}^{2/p} + \omega_{\eps}^{2/p}\right) > \cM_A + 2\eps,
		\]
		getting a contradiction.\\
		Hence, \textit{compactness} holds, namely there exists a sequence $\{\xi_n\}_n$ such that for every $\eps>0$ there exists $R>0$ such that, for every $n \in \N$, 
		\begin{equation}
			\label{Lions:theses}
			\int_{B^c_R(\xi_n)} \mu_n(x) \dd x  \leq\eps.
		\end{equation}
		Let $A$ satisfies \eqref{hyp:A:linear} and consider $v_n(x):=\e^{-iA(\xi_n)\cdot x}u_n(x+\xi_n)$. By \eqref{Lions:theses},
		
		\begin{equation}
			\label{boundedness:outside:ball}
			\text{ for all } \eps > 0 \text{ there exists } R>0 \text{ such that } \sup_{n \in \N}\int_{B^c_R(0)}|v_n(x)|^2 \dd x \leq \eps,
		\end{equation}
		using also Lemma \ref{partial:gauge}, we have that, for every $n \in \N$, $\|v_n\|^2_{A,s} = \|u_n\|^2_{A,s}$ so that $\{v_n\}_n$ is bounded in $H^s_A(\R^3,\C)$, and $\|v_n\|_{L^p(\R^3,\C)} = \|u_n\|_{L^p(\R^3,\C)} = 1$. Then, there exists a function $v \in H^s_A(\R^3,\C)$ such that, up to a subsequence, $v_n \weakto v$ in $H^s_A(\R^3,\C)$, by Lemma \ref{magn:sob:emb} and \eqref{boundedness:outside:ball}, $v_n \to v \text{ in } L^2(\R^3,\C)$, and, by interpolation, also in $L^q(\R^3,\C)$, for every $q \in (2,2^*_s)$, and $\|v\|_{L^p(\R^3,\C)}=1$.\\
		It follows that $v \in \cS$ and 
		\[
		\cM_A \leq \|v\|^2_{A,s} \leq \liminf_n\|v_n\|^2_{A,s} = \liminf_n\|u_n\|^2_{A,s} = \cM_A,
		\]
		providing a solution of \eqref{eq:semilinear}.\\
		If, instead, $A$ satisfies (\hyperlink{Ab}{$\cA_b$}), then $\{\xi_n\}_n$ is bounded. Indeed, if we suppose by contradiction that $\{\xi_n\}_n$ is unbounded and we consider $v_n(x):=\e^{iH_n\cdot x} u_n(x + \xi_n)$, $x \in \R^3$, by Lemma \ref{partial:gauge} we get that, for every $n \in \N$, $v_n \in H^s_{A_n}(\R^3,\C)$ and $[u_n]_{A,s}=[v_n]_{A_n,s}$. By Lemma \ref{precompact:Lq} we can yield that $v_n \to v$ in $L^q_{\loc}(\R^3,\C)$ for every $q \in [1,2^*_s)$.
		Observe that, since $v_n \in H^s_{A_n}(\R^3,\C)$ for every $n$, then by Definition \ref{assumption:A}, $v \in H^s_{A_\Xi}(\R^3,\C)$. Therefore, by Lemma \ref{magn:sob:emb}, $v \in L^q(\R^3,\C)$ for every $q \in [2,2^*_s]$.
		Moreover, by \eqref{Lions:theses}, we get \eqref{boundedness:outside:ball} for this new $\{v_n\}_n$, so that, as before, $v_n \to v$ in $L^q(\R^3,\C)$, for every $q \in [2,2^*_s)$, and $\|v\|_{L^p}=1$.
		By using (\hyperlink{Ab}{$\cA_b$}), Fatou's Lemma, Lemma \ref{partial:gauge}, and the further assumption of this case, we obtain
		\[
		\cM_{A_{\Xi}} 
		\leq 
		\|v\|^2_{A_{\Xi,s}}
		\leq \liminf_n \|v_n\|_{A_n,s}^2
		= \lim_n \|u_n\|_{A,s}^2 = \cM_A < \inf \cM_{A_{\Xi}},
		\]
		reaching a contradiction.\\
		Finally, using the boundedness of $\{\xi_n\}_n$ and \eqref{Lions:theses}, we can repeat for $\{u_n\}_n$ the same arguments used before for $\{v_n\}_n$, obtaining that the weak limit $u$ of $\{u_n\}_n$ in $H^s_A(\R^3,\C)$ belongs to $\cS$ and $\cM_A = \|u\|_{A,s}^2$.
	\end{proof}
	
	\subsection{The Choquard equation}
	\label{subsec:Choquard}
	
	We move our attention on the Choquard equation \eqref{eq:Choquard} for $p \in (1+\alpha/3,(3+\alpha)/(3-2s))$. We begin by noting that it admits a variational structure: weak solutions can be sought as critical points of the functional $\cE_A:H^s_A(\R^3,\C) \to \R$ defined as
	\begin{equation*}
		\begin{aligned}
			\cE(u) := \|u\|_{A,s}^2 - \frac{1}{p}\cD(u), \text{ with } \cD(u):=\int_{\R^3}\left(I_{\alpha} * |u|^p\right)|u|^p \dd x.
		\end{aligned}
	\end{equation*}
	
	Observe that, since in our range of $p$'s, $6p/(3+\alpha)\in (2,2_s^*)$, using Lemma \ref{magn:sob:emb}, by Hardy-Littlewood-Sobolev inequality, there exists $C>0$ such that for all $u\in L^\frac{6p}{\alpha+3}(\R^3,\C)$,
	\begin{equation}
		\label{HLS1}
		\cD(u) \leq C \| u\|_{L^{\frac{6p}{3+\alpha}}(\R^3,\C)}^{2p}
	\end{equation}
	and so, since in our range of $p$'s, $6p/(3+\alpha)\in (2,2_s^*)$, using Lemma \ref{magn:sob:emb}, there is $C>0$ such that for all $u\in H_A^s(\R^3,\C)$,
	\begin{equation}
		\label{HLS2}
		\cD(u) \leq C \| u\|_{A,s}^{2p}.
	\end{equation}
	Moreover, we recall the following Brezis-Lieb type Lemma.
	\begin{Lem}[{\cite[Lemma 2.4]{MoVS}}]
		\label{Brezis:Lieb}
		Let $\alpha \in (0,3)$, $p \in \left[1,6/(3+\alpha)\right)$ and $\{u_n\}_n$ be a bounded sequence in $L^{\frac{6p}{3+\alpha}}(\R^3,\C)$. If $u_n \to u$ almost everywhere in $\R^3$ as $n \to +\infty$, then
		\[
		\lim_n\left(\cD(u_n) - \cD(u_n-u)\right) = \cD(u).
		\]
	\end{Lem}
	As for \eqref{eq:semilinear}, we are going to show the existence of a ground state solutions. 
	We would like to remark that, as described in details in \cite[Section 2.1]{MoVS} when $A=0$, the existence of a ground state solution can be obtained by considering different variational problems:
	direct computations show that such different approaches are, indeed, equivalent in the local case (see \cite[Proposition 2.1]{MoVS}) and in the nonlocal case (see \cite[Proposition 4.1]{dASiSq}).\\
	Here, we show the existence of nontrivial ground state solutions for \eqref{eq:Choquard} by minimizing
	\[
	\cG(u):=\frac{\|u\|^2_{A,s}}{\cD(u)^{\frac{1}{p}}}
	\]
	on $H^s_A(\R^3,\C)\setminus\{0\}$.
	\\
	Let
	\[
	\tau_A:=\inf_{u \in H^s_A(\R^3,\C)\setminus\{0\}}\cG(u).
	\]
	Observe that, by \eqref{HLS2}, $\tau_A>0$, and
	\[
	\tau_A
	=\inf\left\{\|u\|^2_{A,s} : u \in H^s_A(\R^3,\C), \cD(u) = 1 \right\}.
	\]
	
	To prove the existence result in the radial setting, we start considering the restriction of the minimization problem to the radial functions, and so we define 
	\begin{equation*}
		\tau_{A,\rad}:=\inf_{u \in H^s_{A,\rad}(\R^3,\C)\setminus\{0\}}\cG(u) = \inf\left\{\|u\|^2_{A,s} : u \in H^s_{A,\rad}(\R^3,\C), \cD(u) = 1 \right\}.
	\end{equation*}
	In this case, the statement reads as follows.
	\begin{Th}
		\label{Choquard:radial}
		Let $A:\R^3 \to \R^3$ be a vector field with locally bounded gradient, satisfying \eqref{isometry}, and $p \in \left(1+\frac{\alpha}{3},\frac{3+\alpha}{3-2s}\right)$. Then $\tau_{A,\rad}$ is achieved.
	\end{Th}
	\begin{proof}
		Let $\{u_n\}_n \subset H^s_{A,\rad}(\R^3,\C)$ be a minimizing sequence, namely $\|u_n\|^2_{A,s} \to \tau_{A,\rad}$ and $\cD(u_n)=1$, for every $n \in \N$. 
		Then, 
		there exists a function $u \in H^s_{A,\rad}(\R^3,\C)$ such that, up to a subsequence, $u_n \weakto u$ in $H^s_{A,\rad}(\R^3,\C)$. Therefore, by Lemma \ref{radial:comp:emb}, $u_n \to u$ in $L^q(\R^3,\C)$ for every $q \in (2,2^*_s)$ and $u_n \to u$ a.e. in $\R^3$.
		Then, since in our range of $p$'s, $6p/(3+\alpha)\in (2,2_s^*)$, using also Lemma \ref{magn:sob:emb}, \eqref{HLS1}, and Lemma \ref{Brezis:Lieb}, we obtain $\cD(u)=1$.
		Hence,
		\[
		\tau_{A,\rad} 
		\leq \|u\|^2_{A,s} \leq \liminf_n \|u_n\|^2_{A,s} 
		= \tau_{A,\rad}.
		\]
	\end{proof}
	
	Now, we move in the general setting, namely removing the symmetry assumption and we consider $A$ linear. In the proof, we will argue as in \cite{MoVS} (see also \cite{dASiSq}). We get
	\begin{Th}
		\label{Choquard:main:th}
		Suppose $A:\R^3 \to \R^3$ is a vector field which satisfies \eqref{hyp:A:linear}. If $p \in \left(1+\frac{\alpha}{3},\frac{3+\alpha}{3-2s}\right)$, then $\tau_A$ achieves its minimum on $H^s_A(\R^3,\C)$.
	\end{Th}
	\begin{proof}
		Let $\{u_n\}_n \subset H^s_A(\R^3,\C)$ be a minimizing sequence, namely $\|u_n\|^2_{A,s} \to \tau_A$ with $\cD(u_n) = 1$ for every $n \in \N$. Thus, the sequence $\{u_n\}_n$ is bounded in $H^s_A(\R^3,\C)$, and so, by \eqref{HLS1}, \cite[Lemma 2.2]{dASiSq}, and Lemma \ref{diamagnetic:eq}, 
		\begin{align*}
			1&=\cD(u_n) 
			\leq 
			C\|u_n\|^{2p}_{L^{\frac{6p}{3+\alpha}}(\R^3,\C)}
			\leq 
			C\left(\left(\sup_{x \in \R^3}\int_{B_1(x)}|u_n|^{\frac{6p}{3+\alpha}} \dd x \right)^{1-\frac{3+\alpha}{3p}} \||u_n|\|^2_{H^s(\R^3,\R)}\right)^{1+\frac{\alpha}{3}} \\
			&\leq 
			C\left(\left(\sup_{x \in \R^3}\int_{B_1(x)}|u_n|^{\frac{6p}{3+\alpha}} \dd x \right)^{1-\frac{3+\alpha}{3p}} \|u_n\|^2_{A,s}\right)^{1+\frac{\alpha}{3}} 
			\leq
			C\left(\sup_{x \in \R^3}\int_{B_1(x)}|u_n|^{\frac{6p}{3+\alpha}} \dd x \right)^{\left(1-\frac{3+\alpha}{3p}\right)\left(1+\frac{\alpha}{3}\right)}.
		\end{align*}
		Then, since $p>1+\alpha/3$, there exists a sequence $\{x_n\}_n \subset \R^3$ such that
		\begin{equation*}
			\inf_{n \in \N}\int_{B_1(x_n)}|u_n|^{\frac{6p}{3+\alpha}} \dd x > 0.
		\end{equation*}
		Now let us consider $v_n(x):=\e^{-iA(x_n)\cdot x}u_n(x+x_n)$. We have that  $\cD(v_n)=\cD(u_n)=1$ and, since $A$ satisfies \eqref{hyp:A:linear}, as observed in the Proof of Theorem \ref{semilinear:existence:th},
		$\|v_n\|^2_{A,s} = \|u_n\|^2_{A,s}$. Then $\{v_n\}_n$ is bounded in $H^s_A(\R^3,\C)$, $\|v_n\|^2_{A,s} \to \tau_A$, and 
		\begin{equation}
			\label{lions:vn:for:Choquard}
			\inf_{n \in \N}\int_{B_1(0)}|v_n|^{\frac{6p}{3+\alpha}} \dd x > 0.
		\end{equation}
		Therefore, there exists $v \in H^s_A(\R^3,\C)$ such that, up to a subsequence, $v_n \weakto v$ and
		\[
		\|v\|^2_{A,s} = \lim_n\left(\|v_n\|^2_{A,s} - \|v_n-v\|^2_{A,s}\right).
		\]
		In addition, by Lemma \ref{magn:sob:emb}, $v_n \to v$ almost everywhere in $\R^3$, and, using also \eqref{lions:vn:for:Choquard} and that $p < (3+\alpha)(3-2s)$,  $v \neq 0$. Moreover, by Fatou Lemma, 
		\[
		0<\cD(v) \leq \liminf_n\cD(v_n)=1.
		\]
		If $v_n \neq v$, using also Lemma \ref{Brezis:Lieb},
		\begin{equation}
			\label{Sv:tau:equivalence}
			\begin{split}
				\tau_A \cD(v)^{\frac{1}{p}}
				&\leq \cG(v)\cD(v)^{\frac{1}{p}}
				=
				\lim_n\left[\|v_n\|^2_{A,s}-\|v_n-v\|^2_{A,s}\right] \\
				&=\lim_{n}\left[\|u_n\|^2_{A,s} - \cG(v_n-v)\cD(v_n-v)^{\frac{1}{p}} \right]
				=\limsup_{n}\left[\|u_n\|^2_{A,s} - \cG(v_n-v)\cD(v_n-v)^{\frac{1}{p}} \right]\\
				& = \tau_A - \liminf_n \cG(v_n-v)\cD(v_n-v)^{\frac{1}{p}}
				\leq \tau_A \left[1 - \liminf_n \cD(v_n-v)^{\frac{1}{p}} \right]\\
				& = \tau_A \left[1 - (1-\cD(v))^{\frac{1}{p}} \right]
				\leq \tau_A \cD(v)^{\frac{1}{p}}.
			\end{split}
		\end{equation}
		Thus $G(v)=\tau_A$ and so we can conclude. Observe also that $\cD(v) = 1$ since, if $\cD(v) < 1$, the last inequality in \eqref{Sv:tau:equivalence} should be a strict inequality.
	\end{proof}

	\appendix
	\section{Properties of the modified Bessel function}
	\label{Appendix:Bessel}
	
	Before starting with its definition, we recall some useful properties of the modified Bessel functions $\mathcal{K}_\nu$.
	It is well known that they are solutions to the Bessel's differential equations
	\[
	\zeta^2K'' + \zeta K' - (\zeta^2+{\nu}^2)K = 0,
	\]
	they have the following integral representation
	\begin{equation}
		\label{integral:repr:Bessel}
		\bessel{\nu}{\zeta} = \frac12\left(\frac{\zeta}{2}\right)^{\nu}\int_0^{+\infty}\e^{-\frac{\zeta^2}{4t} - t}t^{-1-\nu} \dd t \qquad \text{for } \zeta>0
	\end{equation}
	that actually holds in a more general setting, i.e. for $\zeta \in \C$ with $|\arg \zeta | < \pi/4$ (see \cite[Section 6]{Watson1944}),
	and
	\begin{alignat}{3}
		\label{bessel:positive}
		&\mathcal{K}_\nu (\zeta) > 0 & & \qquad \text{for } \zeta > 0 && \text{ and } \nu \in \R,\\
		\label{Gauntest1}
		&\mathcal{K}_\nu (\zeta)\sim 2^{|\nu|-1} \Gamma(|\nu|) \zeta^{-|\nu|} & & \qquad \text{as } \zeta \downarrow 0, && \text{ for } \nu \neq 0,\\
		\label{Gauntest2}
		&\mathcal{K}_\nu (\zeta)\sim \sqrt{\frac{\pi}{2\zeta}} \e^{-\zeta} & & \qquad \text{as } \zeta\to +\infty, && \text{ for } \nu \in \R,\\
		\label{nunumeno1}
		& \mathcal{K}'_\nu = -\frac{\nu}{\zeta}\mathcal{K}_{\nu} - \mathcal{K}_{\nu-1}.
	\end{alignat}
	Often, in our arguments, for $\nu > 0$, it will be enough to use the rough estimate
	\begin{equation}\label{Bessel:bound}
		0 < \mathcal{K}_{\nu}(\zeta) \leq C\zeta^{-\nu} \qquad \text{for } \zeta > 0,
	\end{equation}
	where $C>0$ depends on $\nu$.\\
	In addition, for $\zeta > 0$ and $\nu \geq 0$, the modified Bessel functions $\mathcal{K}_{\nu}$ are continuous, positive functions of $\nu$ and $\zeta$, and for fixed $\nu$, $\mathcal{K}_{\nu}$ is decreasing in $\zeta$ and, for fixed $\zeta$, we have increasing in $\nu$. For all these properties, we refer to \cite[Chapter 12, Section 1.1]{Olver1997}.

	Furthermore, if we consider $\theta \in H^1(\R^+;\zeta^{1-2s}) := \left\{\eta=\eta(\zeta) \in L^2(\R^+,\R) : \zeta^{1-2s} \eta' \in L^2(\R^+,\R) \right\}$ which solves
	\begin{equation}
		\label{theta:eq}
		\begin{cases}
			\displaystyle\theta'' + \frac{1-2s}{\zeta}\theta' - \theta = 0, \\
			\theta(0)=1,
		\end{cases}
	\end{equation}
	we have that $\theta$ can be expressed in terms of the modified Bessel function, namely 
	\begin{equation*}
		\theta(\zeta)=\frac{2}{\Gamma(s)}\left(\frac{\zeta}{2}\right)^s\mathcal{K}_s(\zeta).
	\end{equation*}
	By \eqref{nunumeno1}, since $s \in (0,1)$, it is easy to see that
	\begin{equation}
		\label{theta:derivative}
		\theta'(\zeta) 
		= -\frac{2^{1-s}}{\Gamma(s)}\zeta^s\mathcal{K}_{s-1}(\zeta).
	\end{equation}
	Moreover, see e.g. \cite{FaFe},
	\begin{equation}
		\label{chain:equivalence}
		\int_0^{+\infty} \zeta^{1-2s}\left(|\theta'(\zeta)|^2 + |\theta(\zeta)|^2\right) \dd \zeta = -\lim_{\zeta \to 0^+}\zeta^{1-2s}\theta'(\zeta) = 2^{1-2s}\frac{\Gamma(1-s)}{\Gamma(s)} := \kappa_s.
	\end{equation}
	
	\section{Proof of Proposition \ref{normHmsC}}\label{Appendix:Prop:6}
	\begin{proof}[Proof of Proposition \ref{normHmsC}]
		The first part is a consequence of the equivalence of the norms \eqref{Hs:norm} and \eqref{Hsm:norm} due to \eqref{symbols:inequalities}.\\ 
		To prove \eqref{Hs:equivalent:norm} let $u \in H^s(\R^3,\C)$ and consider the function $\theta \in H^1(\R^+;\zeta^{1-2s})$ introduced in \eqref{theta:eq}.\\
		First, observe that by \eqref{chain:equivalence},
		\begin{align*}
			&\int_{\R^3}\lim_{t \to 0^+}\frac{\theta\left(\sqrt{|\xi|^2 + m^2}t\right) - 1}{t^{2s}} \left|\cF u(\xi)\right|^2 \dd \xi\\
			& \qquad = \frac{1}{2s}\int_{\R^3}\lim_{t \to 0^+}t^{1-2s}\sqrt{|\xi|^2 + m^2}\theta'\left(\sqrt{|\xi|^2 + m^2}t\right)\left|\cF u(\xi)\right|^2 \dd \xi \\
			& \qquad = \frac{1}{2s}\int_{\R^3}\lim_{t \to 0^+}\left(\sqrt{|\xi|^2 + m^2}t\right)^{1-2s}\theta'\left(\sqrt{|\xi|^2 + m^2}t\right)\left(|\xi|^2 + m^2\right)^s\left|\cF u(\xi)\right|^2 \dd \xi \\
			& \qquad = -\frac{\kappa_s}{2s}\int_{\R^3}\left(|\xi|^2 + m^2\right)^s\left|\cF u(\xi)\right|^2 \dd \xi
			= -\frac{\kappa_s}{2s} \| u\|_{H_m^s(\R^3,\C)}^2.
		\end{align*}
		Now, let $P_m$ be the fundamental solution of
		\begin{equation}
			\label{ext:whole:space}
			-\divv (|t|^{1-2s}\nabla U) + m^2|t|^{1-2s}U = \delta_0, \quad \text{ in } \R^3 \times \R.
		\end{equation}
		Direct computations show that
		\begin{equation}
			\label{Bessel:kernel}
			P_m(x,t)=C'_sm^{\frac{3+2s}{2}}\frac{t^{2s}}{\left(|x|^2+t^2\right)^{\frac{3+2s}{4}}}\mathcal{K}_{\frac{3+2s}{2}}\left(m\sqrt{|x|^2+t^2}\right).
		\end{equation}
		Here, $C'_s$ is the normalizing constant that, by \eqref{Gauntest1}, must satisfy
		\[
		C'_s2^{\frac{1+2s}{2}}\Gamma\left(\frac{3+2s}{2}\right) = p_s,
		\]
		where
		\[
		p_s=\frac{\Gamma\left(\frac{3+2s}{2}\right)}{\pi^{3/2}\Gamma(s)}
		\]
		is the normalizing constant for the \textit{Poisson kernel}\footnote{The fundamental solution of \eqref{ext:whole:space} with $m=0$.}, so that
		\begin{equation}
			\label{constant:Cs'}
			C'_s=\frac{2^{-\frac{1+2s}{2}}}{\pi^{3/2}\Gamma(s)} = \frac{\kappa_s}{2s}C_s.
		\end{equation}
		Moreover, for every fixed $t > 0$, $P_m(\cdot,t)$ is the Fourier transform of the map $\xi \mapsto \theta\left(\sqrt{|\xi|^2 + m^2}t\right)$ and 
		\begin{equation}
			\label{nice:formula:P:theta}
			\int_{\R^3} P_m(x,t) \dd x = \theta(mt).
		\end{equation}
		Thus, given $t>0$, since $P_m(\cdot,t)$ is even and by \eqref{nice:formula:P:theta},
		\begin{align*}
			&\int_{\R^3}\theta\left(\sqrt{|\xi|^2 + m^2}t\right) \left|\cF u(\xi)\right|^2 \dd \xi\\
			&\qquad
			= \int_{\R^3}\cF^{-1} P_m(\cdot,t)(\xi) \overline{\cF u(\xi)} \cF u(\xi)\dd \xi
			= \int_{\R^3}\cF^{-1} P_m(\cdot,t)(\xi) \cF^{-1} \overline{u}(\xi)\cF u(\xi)\dd \xi
			\\
			&\qquad
			= \int_{\R^3}\cF^{-1} \left(P_m(\cdot,t) * \overline{u}\right)(\xi) \cF u(\xi)\dd \xi
			= \int_{\R^3}[P_m(\cdot,t)*\overline{u}](x)u(x)\dd x \\
			& \qquad
			= \frac12\int_{\R^3 \times \R^3}P_m(x-y,t)\overline{u(y)} u(x) \dd x \dd y
			+ \frac12\int_{\R^3 \times \R^3}P_m(y-x,t) \overline{u(x)} u(y)\dd x \dd y \\
			& \qquad
			= \frac12\int_{\R^3 \times \R^3}P_m(x-y,t)\left(\overline{u(y)} u(x) + \overline{u(x)} u(y) \right)\dd x \dd y\\
			& \qquad
			= -\frac{1}{2}\int_{\R^3\times\R^3} P_m (x-y,t)|u(x)-u(y)|^2 \dd x \dd y
			+ \frac{1}{2}\int_{\R^3\times\R^3} P_m(x-y,t)|u(x)|^2 \dd x \dd y \\
			& \qquad \qquad
			+ \frac{1}{2}\int_{\R^3\times\R^3} P_m(x-y,t)|u(y)|^2 \dd x \dd y \\
			& \qquad
			= -\frac{1}{2}\int_{\R^3\times\R^3} P_m (x-y,t)|u(x)-u(y)|^2 \dd x \dd y 
			+ \int_{\R^3\times\R^3} P_m(x-y,t)|u(x)|^2 \dd x \dd y\\
			& \qquad
			= -\frac{1}{2}\int_{\R^3\times\R^3} P_m (x-y,t)|u(x)-u(y)|^2 \dd x \dd y 
			+ \int_{\R^3} \left( \int_{\R^3} P_m(x-y,t) \dd y\right) |u(x)|^2 \dd x\\
			& \qquad 
			= -\frac{1}{2}\int_{\R^3\times\R^3}P_m(x-y,t)|u(x)-u(y)|^2\dd x \dd y
			+ \theta(tm) \int_{\R^3}|u(x)|^2\dd x.
		\end{align*}
		Hence, using also \eqref{Bessel:kernel}, we obtain
		\begin{equation}
			\begin{split}
				\label{intermediate:theta:passage}
				&\int_{\R^3}\frac{\theta\left(\sqrt{|\xi|^2 + m^2}t\right) - 1}{t^{2s}} \left|\cF{u}(\xi)\right|^2 \dd \xi \\
				&= \frac{\theta(tm)-1}{t^{2s}}\|u\|^2_{L^2(\R^3,\C)} - \frac{C'_s}{2}m^{\frac{3+2s}{2}}\int_{\R^3 \times \R^3}\frac{|u(x)-u(y)|^2}{\left(t^2+|x-y|^2\right)^{\frac{3+2s}{4}}} \mathcal{K}_{\frac{3+2s}{2}}\left(m\sqrt{t^2+|x-y|^2}\right)\dd x \dd y.
			\end{split}
		\end{equation}
		Observe that, by \eqref{theta:derivative} and \eqref{Gauntest1},
		\begin{equation}
			\label{finitness:of:limit}
			\lim_{t \to 0^+}\frac{\theta(t)-1}{t^{2s}} = \frac{1}{2s}\lim_{t \to 0^+} \frac{\theta'(t)}{t^{2s-1}} = -\frac{1}{2^ss\Gamma(s)} \lim_{t \to 0^+}\frac{\bessel{s-1}{t}}{t^{s-1}} = -\frac{\Gamma(1-s)}{2^{2s}s\Gamma(s)} = -\frac{\kappa_s}{2s}.
		\end{equation}
		By \eqref{Bessel:bound}, for a.e. $x,y \in \R^3$ and $t>0$,
		\begin{align*}
			0 < m^{\frac{3+2s}{2}}\frac{|u(x)-u(y)|^2}{\left(t^2+|x-y|^2\right)^{\frac{3+2s}{4}}} \mathcal{K}_{\frac{3+2s}{2}}\left(m\sqrt{t^2+|x-y|^2}\right) &
			\leq C\frac{|u(x)-u(y)|^2}{\left(t^2+|x-y|^2\right)^{\frac{3+2s}{2}}} \\
			&
			\leq C\frac{|u(x)-u(y)|^2}{|x-y|^{3+2s}} \in L^1(\R^3 \times \R^3,\R),
		\end{align*}
		since $u \in H^s(\R^3,\C)$.\\
		Thus, by \eqref{intermediate:theta:passage}, \eqref{finitness:of:limit}, and \eqref{constant:Cs'},
		\begin{align*}
			&\lim_{t \to 0^+} \int_{\R^3}\frac{\theta\left(\sqrt{|\xi|^2 + m^2}t\right) - 1}{t^{2s}} \left|\cF u(\xi)\right|^2 \dd \xi \\
			&\quad = -\frac{\kappa_s}{2s}\left[m^{2s}\|u\|^2_{L^2(\R^3,\C)} + \frac{C_s}{2}m^{\frac{3+2s}{2}}\int_{\R^3 \times \R^3} \frac{|u(x)-u(y)|^2}{|x-y|^{\frac{3+2s}{2}}} \mathcal{K}_{\frac{3+2s}{2}}\left(m|x-y|\right)\dd x \dd y\right].
		\end{align*}
		Finally, since, by \eqref{finitness:of:limit} and by \eqref{Bessel:bound}, the function $0<t \mapsto (\theta(t)-1)/t^{2s}$ is bounded, then, using \eqref{Hsm:norm}, we have
		\[
		\left|\frac{\theta\left(\sqrt{|\xi|^2 + m^2}t\right)-1}{t^{2s}} \left|\cF{u}(\xi)\right|^2\right|
		\leq C\left(|\xi|^2 + m^2\right)^s \left|\cF u(\xi)\right|^2 \in L^1(\R^3).
		\]
		Hence we get
		\begin{equation*}
			\lim_{t \to 0^+} \int_{\R^3}\frac{\theta\left(\sqrt{|\xi|^2 + m^2}t\right) - 1}{t^{2s}} \left|\cF u(\xi)\right|^2 \dd \xi = \int_{\R^3}\lim_{t \to 0^+} \frac{\theta\left(\sqrt{|\xi|^2 + m^2}t\right) - 1}{t^{2s}} \left|\cF{u}(\xi)\right|^2 \dd \xi
		\end{equation*}
		concluding the proof of \eqref{Hs:equivalent:norm}.
	\end{proof}

	\section*{Acknowledgements}
	The authors are members of GNAMPA (INdAM) and have been partially supported by PRIN 2017JPCAPN {\em Qualitative and quantitative aspects of nonlinear PDEs}. Federico Bernini is partially supported by GNAMPA Project 2024 {\em Problemi spettrali e di ottimizzazione di forma: aspetti qualitativi e stime quantitative}. Pietro d'Avenia is partially supported by European Union - Next Generation EU - PRIN 2022 PNRR P2022YFAJH {\em Linear and Nonlinear PDE’s: New directions and Applications}, by the Italian Ministry of University and Research under the Program Department of Excellence L. 232/2016 (Grant No. CUP D93C23000100001), and by GNAMPA Project 2024 {\em Metodi variazionali e topologici per alcune equazioni di Schr\"odinger nonlineari}.

\end{document}